\documentclass[11pt]{article}
\def\x{2.9}
\usepackage[lmargin=\x cm,rmargin=\x cm,tmargin=\x cm,bmargin=\x cm,bindingoffset=0cm]{geometry}
\usepackage{amsmath}
\usepackage{amssymb}
\usepackage{amsthm}
\usepackage{amsfonts}
\usepackage{float}
\usepackage{graphicx} 
\usepackage{tikz}
\usepackage{color}
\usepackage{multicol} 
\usepackage{xargs}
\usepackage{nicefrac}
\usepackage{bbm}
\usepackage{mathrsfs}
\usepackage[english]{babel}
\usepackage[latin1]{inputenc}
\usepackage{hyperref}

\title{Automorphic Lie Algebras and Cohomology of Root Systems}
\date{}
\author{Vincent Knibbeler$^{a,b}$, Sara Lombardo$^{a}$ and Jan A. Sanders$^{b}$ \\
\\
$^{a}$Department of Mathematical Sciences, School of Science\\
Loughborough University, Schofield Building, LE11 3TU, Loughborough, UK\\
\and
\\
$^{b}$Department of Mathematics, Faculty of Sciences\\
Vrije Universiteit, De Boelelaan 1081a, 1081 HV Amsterdam, The Netherlands
}

\newcommand{\roots}{\Phi}
\newcommand{\sroots}{\Delta}
\newcommand{\rootp}{\pmb{+}}
\newcommand{\rootm}{\pmb{-}}
\newcommand{\rootpm}{\pmb{\pm}}
\newcommand{\rootmp}{\pmb{\mp}}

\newcommand{\mb}[1]{\mathbbm{#1}}
\newcommand{\mero}[1][\cp1]{\mathcal{M}(#1)}
\newcommand{\mf}[1]{\mathfrak{#1}}
\newcommand{\Zn}[1]{\bbbz/{#1}\bbbz}

\newcommand{\Aut}{\mathrm{Aut}}

\newcommand{\inp}[2]{(#1\cbr\, #2)}

\def\cp1{{\mathbbm C\mathbb P}^1}

\def\bbbz{\mathbbm Z}
\def\ad{{\mathrm{ad}}}

\def\Mon{\mathsf{M}}

\def\lbr{{[\,}}
\def\br{{|\,}}
\def\rbr{{\,]}}
\def\ibr{{|\,}}
\def\cbr{{,\,}}

\newcommand\colorchain{gray}
\newcommand\colorp{red}
\newcommand\colorm{blue}
\tikzstyle{root}=[scale=0.9,shape=circle ]
\tikzstyle{chain}=[color=\colorchain,->,shorten >=8pt,shorten <=8pt, >=stealth]
\tikzstyle{schain}=[color=\colorchain,-,shorten >=4pt,shorten <=4pt]
\tikzstyle{cochainp}=[thick, color=\colorp,->,shorten >=8pt,shorten <=8pt, >=stealth]
\tikzstyle{cochainm}=[thick, color=\colorm,->,shorten >=8pt,shorten <=8pt, >=stealth]
\tikzstyle{scochainp}=[thick, color=\colorp,-,shorten >=4pt,shorten <=4pt]
\newcommand{\scaleAA}{0.5}

\newtheorem{Theorem}{Theorem}[section]

\newtheorem{Lemma}[Theorem]{Lemma}
\newtheorem{Corollary}[Theorem]{Corollary}

\newtheorem{Definition}[Theorem]{Definition}
\newtheorem{Example}[Theorem]{Example}

\newtheorem{Remark}[Theorem]{Remark}

\begin{document}

\maketitle
\begin{abstract}
A cohomology theory of root systems emerges naturally in the context of Automorphic Lie Algebras, where it helps formulating some structure theory questions. In particular, one can find concrete models for an Automorphic Lie Algebra by integrating cocycles.
In this paper we define this cohomology and show its connection with the theory of Automorphic Lie Algebras. Furthermore, we discuss its properties: we define the cup product, we show that it can be restricted to symmetric forms, that it is equivariant with respect to the automorphism group of the root system, and finally we show acyclicity at dimension two of the symmetric part, which is exactly what is needed to find concrete models for Automorphic Lie Algebras.

Furthermore, we show how the cohomology of root systems finds application beyond the theory of Automorphic Lie Algebras by applying it to the theory of contractions and filtrations of Lie algebras. In particular, we show that contractions associated to Cartan $\mb{Z}$-filtrations of simple Lie algebras are classified by $2$-cocycles, due again to the vanishing of the symmetric part of the second cohomology group.

\vspace{1cm}
\\17B22  	Lie algebras and Lie superalgebras: Root systems;
\\17B65  	Lie algebras and Lie superalgebras: Infinite-dimensional Lie (super)algebras

\vspace{1cm}

\vspace{1cm}
\noindent\textit{Acknowledgements}\\
This work has been initially supported by grants from EPSRC (EP/E044646/1 and \\EP/E044646/2) and NWO (VENI 016.073.026). 
S L  gratefully acknowledges hospitality from the Isaac Newton Institute for Mathematical Science in Cambridge, UK, during September 2019, where the final revision of this paper was completed.

\end{abstract}
\section{Introduction: Automorphic Lie Algebras and their models}
\label{sec:intro}
The purpose of this paper is to set up a cohomology theory for root systems \(\roots\). The motivation comes from the theory of Automorphic Lie Algebras \cite{Lombardo,lm_cmp05,knibbeler2014}, a generalisation of current algebras that arose in the context of integrable systems, which also provide a natural application. The theory itself is however very general, and potentially could reach out to many branches of mathematics. Indeed, while it can be seen as a branch of groupoid cohomology, its novelty comes from the fact that it is constructed on root systems \(\roots\) and thus benefits from their many properties.  

For the purpose of this paper, Automorphic Lie Algebras, or ALiAs, can be defined as follows. 
For a subset $\Gamma$ on the projective line $\mb{CP}^1$, let $\mero_\Gamma$ denote the meromorphic functions on $\mb{CP}^1$ whose poles are contained in $\Gamma$.  
Let $G$ be a finite group of automorphisms of $\mb{CP}^1$ and assume $G$ also acts on a complex Lie algebra $\mf{g}$ by Lie algebra automorphisms. Then, for any $G$-stable set $\Gamma\subset\mb{CP}^1$, $G$ acts canonically on the Lie algebra $\mf{g}\otimes\mero_\Gamma$ by Lie algebra automorphisms (the Lie bracket of $\mf{g}\otimes\mero_\Gamma$ is taken to be the $\mero$-linear extension of the bracket of $\mf{g}$). The Lie subalgebra $\mf{A}$ consisting of pointwise 
fixed elements is called an Automorphic Lie Algebra \cite{Lombardo,lm_cmp05}
\[
\mf{A}=\left(\mf{g}\otimes\mero_\Gamma\right)^G\,.
\] 
Originally defined in the context of integrable systems (see, e.g. \cite{Lombardo,lm_cmp05}) where $G$ is called a \emph{reduction group} \cite{Mikhailov81,lm_jpa04,lm_cmp05}, they are rather universal equivariant objects, arising in different and seemingly disconnected contexts.
As a matter of fact, they are also known by the name Equivariant Map Algebras, following a completely independent development (see, e.g. \cite{NS_2011}). Indeed, $\mf{A}$ consists of all $G$-equivariant meromorphic maps $\mb{CP}^1\rightarrow\mf{g}$ with poles confined to $\Gamma$.

The following example illustrates the simplest ALiA, where $\mf{g}=\mf{sl}_2(\mb{C})$ and $G=\Zn{2}$; we refer the interested reader to \cite{MR2718933,KLS15} for details, more examples and classification results.

\begin{Example}
\label{ex:simpleALIA}
Let $\mf{sl}_2(\mb{C})$ be represented by complex $2\times 2$ matrices with zero trace. 
If elements of $\cp1$ are written as $[X:Y]$, where $(X,Y)\in\mb{C}^2\setminus\{(0,0)\}$ and $[X:Y]=[cX:cY]$ for all $c\in\mb{C}\setminus\{0\}$, then meromorphic functions $\mero$ can be identified with quotients of homogeneous polynomials in the variables $X$ and $Y$ of identical degree.
Let the action of the reduction group $G=\Zn{2}$ on the projective line be given by 
\[[X:Y]\mapsto [Y:X].\] 
The two orbits with nontrivial isotropy are singletons $\Gamma_i=\{[1:1]\}$ and $\Gamma_j=\{[-1:1]\}$. 
Let $\Gamma$ be $\{[0:1],[1:0]\}$ (this orbit has trivial isotropy) and construct automorphic functions $\mb{I}$ and $\mb{J}$ to vanish on $\Gamma_i$ and $\Gamma_j$, respectively, and to have poles in $\Gamma$: 
\[\mb{I}=-\frac{(X-Y)^2}{4XY},\quad \mb{J}=\frac{(X+Y)^2}{4XY},\quad\mb{I}+\mb{J}+1=0.\] 
The ring of automorphic functions is $\mero_\Gamma^{\Zn{2}}=\mb{C}[\mb{I}]$.

Consider next a $\Zn{2}$ action on the base Lie algebra, defined by
\[\begin{pmatrix}a&b\\c&-a\end{pmatrix}\mapsto \begin{pmatrix}1&0\\0&-1\end{pmatrix}\begin{pmatrix}a&b\\c&-a\end{pmatrix}\begin{pmatrix}1&0\\0&-1\end{pmatrix}^{-1}= \begin{pmatrix}a&-b\\-c&-a\end{pmatrix}.\]
Then, the Automorphic Lie Algebra $\mf{A}=\left(\mf{sl}_2(\mb{C})\otimes\mero_\Gamma\right)^{\Zn{2}}$ is generated as a $\mb{C}[\mb{I}]$-module by
\[
{h}=\begin{pmatrix}1&0\\0&-1\end{pmatrix},\quad
{e}=\frac{i(X-Y)(X+Y)}{4XY}\begin{pmatrix}0&1\\0&0\end{pmatrix},\quad
{f}=\frac{i(X-Y)(X+Y)}{4XY}\begin{pmatrix}0&0\\1&0\end{pmatrix},
\]
where the constant factor $i/4$ is used to obtain the Lie bracket relations
\[
[h,e]=2e,\quad[h,f]=-2f,\quad[e,f]=\mb{I}\mb{J}\,h.
\]
From these Lie brackets it is clear that the Lie subalgebra of $\mf{sl}_2(\mb{C}[\mb{I}])$ generated as a $\mb{C}[\mb{I}]$-module by \[H=\begin{pmatrix}1&0\\0&-1\end{pmatrix},\quad E=\begin{pmatrix}0&\mb{IJ}\\0&0\end{pmatrix},\quad F=\begin{pmatrix}0&0\\1&0\end{pmatrix}\] is isomorphic to $\mf{A}$, and is arguably simpler to work with than the original representation of $\mf{A}$. Observe that for larger reduction groups $G$, the matrices of invariants $H,E,F$ provide a tremendous simplification relative to the invariant matrices $h,e,f$. This \emph{model} is not unique, as the functions $\mb{I}$ and $\mb{J}$ can be placed in either of the nilpotent generators. The model can be seen as a function $\omega^1$ from the three eigenvalues $\{0,2,-2\}$ of $\ad(h)$ to the space of monomials in $\mb{I}$ and $\mb{J}$. The monomials in the structure constants take the form $\mathsf{d}^1\omega^1(\alpha,\beta)=\frac{\omega^1(\alpha)\omega^1(\beta)}{\omega^1(\alpha+\beta)}$, $\alpha,\beta\in\{0,2,-2\}$, and it is precisely the kernel of $\mathsf{d}^1$ which describes the freedom in choosing a model for $\mf{A}$.
\end{Example}
\begin{Remark}
In the previous example, we immediately find a Cartan-Weyl basis for the Automorphic Lie algebra. 
This is due to a convenient property of cyclic group actions on simple Lie algebras by inner automorphisms, namely $\mf{g}^{\Zn{n}}$ contains a Cartan subalgebra (CSA) of the simple base Lie algebra $\mf{g}$. Consequently, the root spaces of $\mf{g}$ with respect to this CSA are preserved by $\Zn{n}$. This is not the case for general reduction groups $G$ and in fact obtaining a Chevalley normal form is a significant computational challenge and proving its existence in an abstract setting a major theoretical challenge.
\end{Remark}
As illustrated in Example \ref{ex:simpleALIA}, in the theory of ALiAs one can compute a Cartan-Weyl basis \cite{MR2718933,KLS15} and find that the structure constants are monomials in at most three variables, say \(\mb{I}\), \(\mb{J}\) and \(\mb{K}\) (which are $G$-invariant functions, each vanishing precisely on one of the $G$-orbits in $\mb{CP}^1$ that have nontrivial isotropy, and satisfy \(\mb{I}+\mb{J}+\mb{K}=0\)).
An immediate question is whether one can find a \emph{model} for such a Lie algebra by taking the Weyl-Chevalley normal form of a classical Lie algebra over \(\mb{C}\) and multiply the weight vectors by monomials in \(\mb{I}\), \(\mb{J}\) and \(\mb{K}\),
as in Definition \ref{def:oneform}. 
Such a model can help us to solve the important question of whether different models of ALiAs are equivalent under the action of \(\Aut(\roots)\), automorphisms of the root system.
It can also provide a useful tool to analyse a more general isomorphism question, namely whether two given ALiAs are isomorphic.

The model can be seen as a 1-form on the root system, in such a way that the coboundary operator \(\mathsf{d}^1\) of this 1-form determines the structure constants of the ALiA.
A natural question is whether every ALiA admits such a model: if the second cohomology is zero, this is indeed the case.
The proof that the second cohomology is zero for root systems of simple Lie algebras is a key-result of this paper.
Since it is entirely constructive, the proof also provides an integration procedure, allowing one to find a model from the given ALiA.

The construction of a Lie algebra with parameters can also be done without reference to ALiAs, by changing the field to monomials with complex coefficients, that is to say, to each weight vector \(e_\alpha\) one
assigns a monomial. This description is worked out in this section
and may serve as an example for the reader when we move on to the abstract theory in the subsequent sections.

\begin{Definition}
\label{def:realisation of root system}
Let \(C^2_\wedge(\roots,\mb{Z})\) be the space of skew 2-forms with arguments in the root system \(\roots\) and values in \(\mb{Z}\), where `skew' stands here for the property $\omega(\beta,\alpha)=-\omega(\alpha,\beta)$.
It is well known (e.g. \cite[Section 25.2]{MR0323842}, \cite[Introduction]{tits1966}) that the bracket relations of a semisimple Lie algebra \(\mf{g}\) over \(\mb{C}\) can be written in terms of a Cartan-Weyl basis \(\langle e_\alpha, e_{\rootm\alpha},h_r\rangle_{\alpha\in\roots^+,\, i=1,\ldots,\ell}\), 
where $\roots^+$ is a set of positive roots,
in which the commutation relations are: 
\begin{align*}
{[}h_r,h_{s}]&=0,\\
{[}h_r, e_\alpha]&=\alpha(h_r)e_\alpha,\\
{[} e_\alpha, e_\beta]&=\varepsilon^2(\alpha,\beta) e_{\alpha\rootp \beta},\quad \varepsilon^2\in C_\wedge^2(\roots,\mb{Z}),\quad \alpha\rootp\beta\in\roots,\\
{[}e_\alpha,e_{\rootm\alpha}]&=\varepsilon^2(\alpha,\rootm\alpha) h_\alpha.
\end{align*}
We use the \(\rootp, \rootm\) symbol for addition, resp. subtraction of the roots \(\alpha\in\roots\). 
If \(\alpha=\sum_{i=1}^\ell m_i \alpha_i\), \(\alpha_i\in\sroots\), where \(\sroots\) denotes the set of simple roots, then \(h_\alpha=\sum_{i=1}^\ell m_i h_i\).
We scale \(e_{\rootm\alpha}\) such that \(\varepsilon^2(\alpha,\rootm\alpha)=1, \,\alpha\in\sroots\).
\end{Definition}
For possible choices of the values of \(\varepsilon^2\) see \cite{tits1966}, and see \cite[Section 03]{frenkel1980basic} and \cite{MR1104219} for the early cohomological interpretation of it.

\begin{Remark}
\label{rem:antisymmetry}
Notice that in \(C^2_\wedge(\roots,\mb{Z})\) the skew symmetry condition is $\omega(\beta,\alpha)=-\omega(\alpha,\beta)$, while in $C^{2}_{-}(\roots,M)$ -- see Section \ref{sec:reversalsymmetry} --  the antisymmetry condition will be defined by $\omega(\beta,\alpha)=\omega(\alpha,\beta)^{-1}$.
\end{Remark}

We now change the field \(\mb{C}\) in the Lie algebra by allowing the coefficients \(\varepsilon^2(\alpha,\beta)\) to be multiplied by monomials in a finite number of variables and call the new Lie algebra
\(\tilde{\mf{g}}\).
An Automorphic Lie Algebra is of this form if its Cartan subalgebra equals that of the base Lie algebra with the field extended to the ring of automorphic functions (cf. \cite{knibbeler2014}).

\begin{Definition}
\label{def:Mon}
Let \(\Mon(\mb{I}_1,\cdots,\mb{I}_k)=\{\mb{I}_1^{m_1} \cdots \mb{I}_k^{m_k}\;|\;m_1,\ldots,m_k\in \mb{Z}\}\) be the {\em rational monomials}, 
and let \(\Mon[\mb{I}_1,\cdots,\mb{I}_k]=\{\mb{I}_1^{m_1} \cdots \mb{I}_k^{m_k}\;|\;m_1,\ldots,m_k\in \mb{N}_0\}\) be the {\em natural monomials}.  
The set of rational monomials (resp. natural monomials) inherits an additive and multiplicative structure from \(\mb{Z}\) (resp. \(\mb{N}\)) by identifying a monomial with its power.
For instance: one adds \(\mb{I}_1^4 \mb{I}_2^2 \) to \(\mb{I}_1^3 \) to obtain \(\mb{I}_1^7\mb{I}_2^2\) and multiplies them to obtain \(\mb{I}_1^{12}\). 
So we will identify \(\Mon(\mb{I}_1,\cdots,\mb{I}_k)\) with \(\mb{Z}^k\) and \(\Mon[\mb{I}_1,\cdots,\mb{I}_k]\) with \(\mb{N}_0^k\).
\end{Definition}
We define a symmetric 2-form  \(\omega_+^2\in  C_+^2(\roots,\Mon(\mb{I}_1,\cdots,\mb{I}_k))\), where the \(+\) indicates the symmetry \(\omega_+^2(\alpha,\beta)=\omega_+^2(\beta,\alpha)\). 
Together with the previous antisymmetric 2-form 
\(\varepsilon^2\in C_\wedge^2(\roots,\mb{Z})\), 
we want it to define a Lie algebra by
\begin{align}
\begin{split}\label{eq:genCartanWeyl} 
{[}h_r,h_{s}]&=0,\\
{[}h_r, e_\alpha]&=\alpha(h_r)e_\alpha,\\
{[} e_\alpha, e_\beta]&=\varepsilon^2(\alpha,\beta) \omega_+^2(\alpha,\beta) e_{\alpha\rootp \beta},\quad  \alpha\rootp\beta\in\roots,\\
{[}e_\alpha,e_{\rootm\alpha}]&=\varepsilon^2(\alpha,\rootm\alpha)\omega_+^2(\alpha,\rootm\alpha) h_\alpha.
\end{split}
\end{align}
For this to happen, the terms of the Jacobi identity must all three have the same monomial in
\(\Mon(\mb{I}_1,\cdots,\mb{I}_k)\), i.e. computing modulo multiplication with integers one has:
\begin{align*}
{[[}e_\alpha,e_\beta],e_\gamma]&\equiv \omega_+^2(\alpha,\beta)\omega_+^2(\alpha\rootp\beta,\gamma)e_{\alpha\rootp\beta\rootp\gamma},\\
{[}e_\alpha,[e_\beta,e_\gamma]]&\equiv \omega_+^2(\alpha,\beta\rootp\gamma)\omega_+^2(\beta,\gamma)e_{\alpha\rootp\beta\rootp\gamma},\\
{[}e_\beta,[e_\alpha,e_\gamma]]&\equiv \omega_+^2(\beta,\alpha\rootp\gamma)\omega_+^2(\alpha,\gamma)e_{\alpha\rootp\beta\rootp\gamma},
\end{align*}
where the multiplication among the \(\omega_+^2\)s is done according to the {\em addition} rules as given in Definition \ref{def:Mon} for \(\Mon(\mb{I}_1,\cdots,\mb{I}_k)\).
We now define 
\[
\mathsf{d}^2\omega_+^2(\alpha,\beta,\gamma)=\frac{\omega_+^2(\beta,\gamma) \omega_+^2(\alpha,\beta\rootp\gamma)}{\omega_+^2(\alpha,\beta)\omega_+^2(\alpha\rootp\beta,\gamma)},
\]
where $\mathsf{d}^2$ is a coboundary operator -- see Section \ref{sec:differential} for this choice of coboundary notation.\\
If \(\mathsf{d}^2\omega_+^2(\alpha,\beta,\gamma)=1\) for all \(\alpha, \beta, \gamma\in \roots\) such that \(\alpha\rootp\beta\), \(\beta\rootp\gamma\) and \( \alpha\rootp\beta\rootp\gamma\) exist,
then the three monomials in the Jacobi identity are equal, and the Jacobi identity is satisfied (by the Jacobi identity of the underlying Lie algebra associated to $\roots$). 
Vice versa, if the Jacobi identity is satisfied, then \(\mathsf{d}^2\omega_+^2(\alpha,\beta,\gamma)=1\). Indeed, for every triple \(\alpha, \beta, \gamma\in \roots\) such that \(\alpha\rootp\beta\), \(\beta\rootp\gamma\) and \( \alpha\rootp\beta\rootp\gamma\) are roots, the two terms ${[[}e_\alpha,e_\beta],e_\gamma]$ and ${[}e_\alpha,[e_\beta,e_\gamma]]$ in the Jacobi identity are nonzero. Therefore their monomial powers are identical, i.e. \(\mathsf{d}^2\omega_+^2(\alpha,\beta,\gamma)=1\).

Let $Z_+^2(\roots,\Mon(\mb{I}_1,\cdots,\mb{I}_k))$ be the the group of $2$-cocycles. We have just proven the following theorem.
\begin{Theorem}
Any \(\omega_+^2\in Z_+^2(\roots,\Mon(\mb{I}_1,\cdots,\mb{I}_k))\) determines a Lie algebra with monomial coefficients of the form (\ref{eq:genCartanWeyl}). 
\end{Theorem}
\begin{Definition}\label{def:oneform}
Let \(\omega^1\in C^1(\roots,\Mon(\mb{I}_1,\cdots,\mb{I}_k))\) and
\(\omega^2_+(\alpha,\beta)=\mathsf{d}^1\omega^1(\alpha,\beta)=\frac{\omega^1(\alpha)\omega^1(\beta)}{\omega^1(\alpha\rootp\beta)}\). We say that \(\omega^1\) is a \emph{model} for \(\omega_+^2\). 
If we take a Cartan-Weyl basis of a Lie algebra and define \(\tilde{e}_\alpha=\omega^1(\alpha) e_\alpha\),
we again have a Lie algebra with an identical Cartan matrix.
\end{Definition}
The previous definition is justified by the following corollary.
\begin{Corollary}
\label{cor:model}
Let $\omega^1$ be defined as above. Then \(\mathsf{d}^1\omega^1\in\) 
\(Z_+^2(\roots,\Mon(\mb{I}_1,\cdots,\mb{I}_k))\)
and satisfies the requirements to define a Lie algebra. Furthermore, \(K_{\tilde{\mf{g}}}( \tilde{e}_\alpha,\tilde{e}_{\rootm\alpha})=\mathsf{d}^1\omega^1(\alpha,\rootm\alpha)K_{\mf{g}}(e_\alpha,e_{\rootm\alpha})\), \(\alpha\in\roots^+\), where \(K_{\tilde{\mf{g}}}\) and \(K_{\mf{g}}\) are the Killing forms of the Lie algebras \(\tilde{\mf{g}}\) and \(\mf{g}\), respectively.
\end{Corollary}
One of the fundamental questions is thus whether there is always a model. This is equivalent to the question whether the second cohomology group \(H_+^2(\roots,\Mon(\mb{I}_1,\cdots,\mb{I}_k))\) is trivial -- see Section \ref{sec:acyclicity}.

This paper is organised as follows: in Section \ref{sec:CohomRoot} we introduce the root systems cohomology. We define in particular $n$-links and $n$-chains as well as (co)boundary operators. We then introduce  a number of concepts, namely cup product (Section \ref{sec:cupproduct}), reversal symmetry (Section \ref{sec:reversalsymmetry}) and equivariance with respect to automorphisms of the root system (Section \ref{sec:equivariance}), before going to our main result, the acyclicity of the second cohomology in Section \ref{sec:acyclicity}.
These sections, in particular Section \ref{sec:cupproduct}, are presented mainly to demonstrate that the theory is complete, and possesses all the desired features. The cup product, for example, provides a strong motivation for cohomology theory (in comparison with homology theory), since its product structure allows one to look for ring-generators of low degree. While it does not have an immediate application, it can be defined also in this context, as we show in Section \ref{sec:cupproduct}. The reversal symmetry in Section \ref{sec:reversalsymmetry} is rather fundamental
and it allows us to systematically develop the theory. 
From the point of view of applications, Section \ref{sec:acyclicity} contains a fundamental result, namely the triviality of \(H_+^2(\roots,\Mon(\mb{I}_1,\cdots,\mb{I}_k))\), and it also contains some illustrative examples (Section \ref{sec:examples}). 
In Section \ref{sec:diagonal contractions and cartan filtrations} we discuss a different application of the cohomology theory, namely we apply it to contractions and filtrations of Lie algebras, where the acyclicity at dimension two plays a fundamental role as well. The final section contains open problems and perspectives for future work.

\section{ Cohomology of root systems} 
\label{sec:CohomRoot}
Let \(\roots\) be a reduced root system, that is a root system satisfying the additional property that the only multiples of a root \(\alpha\) in \(\roots\) are $\pm \alpha$, 
and let $\roots_0=\roots\cup\{0\}$. Let $\roots^+$ be a subset of positive roots and let $\sroots$ be the set of the corresponding simple roots, that is, positive roots which cannot be written as the sum of two elements of $\roots^+$. We denote the addition in the root system by \(\rootp\) to discern it from the formal addition (denoted by $+$) used in the definition of chains that is to follow.
We define  $n$-links \(T_n(\roots)\) inductively as the set of expressions 
\begin{align*}
&\lbr r_1 \br r_2\br\cdots\br r_n\rbr,\quad r_j \in\roots_0,\quad j=1,\ldots,n,\quad\text{such that}
\\ 
&\lbr r_1 \br \cdots \br r_{i-1} \br r_i\rootp r_{i+1} \br r_{i+2} \br\cdots \br r_n \rbr\in T_{n-1}(\roots), \quad i=1,\ldots,n-1,
\end{align*}
with \(T_1(\roots)=\roots_0\). 
This is done to avoid the possibility that in applying the boundary operator as defined in Section \ref{sec:differential}, 
we try to add roots the sum of which is not a root (and this
again is related to the fact that for such roots the bracket of their weight vectors will be zero). In the formula above the link is zero (denoted by \(0_n\in T_n(\roots)\)) if one of its constituents is zero, as happens when opposite roots are added.
The old-fashioned \(\br\)-notation (cf \cite{MR0056295}, later replaced by \(\otimes\)) is used here because the modern notation would seem to imply things like linearity with respect to \(\rootp\) which is not at all the case. The \(\lbr\) and \(\rbr\) are added for readability, usually when \(n>1\) but also sometimes in expressions like \(\lbr r_0\rootp r_1 \rbr\).

Although we will use the same formulas as in group cohomology theory in the definition of the (co)boundary operators, it should be noticed that the root system \(\roots_0\) is not a group, but a groupoid.
Thus root cohomology is a branch of groupoid cohomology. 
Our main result, Theorem \ref{thm:H^2_+=0}, relies on properties of root systems and is therefore not likely to appear in a general treatment of groupoid cohomology. For future research on ALiAs, such as the study of their representations or the investigation of ALiAs with non-inner reduction group, it is also interesting to generalise the root system to the groupoid of weights of representations of semisimple Lie algebras (in which case weight cohomology would be a more appropriate name).

Let the set of $n$-chains \(C_n(\roots)\) be the \(\mb{Z}\)-linear span of the links \(T_n(\roots)\) and the $n$-cochains the maps \(C^n(\roots,M)=C^1(C_n(\roots),M)\), where \(C^1(\square,M)\) are the \(\bbbz\)-linear 1-forms
with values in the \(\mb{Z}\)-module \(M=R^k\), where \(R\) is a ring. Notice that in our applications to ALiAs (e.g. see Section \ref{sec:examples}) we would rather like to restrict to a semiring, \(\mb{N}_0\), but for the cohomological computations we will need a ring to compute in.
\begin{Example}
\label{ex:chains on A2}
Let $\roots$ be the root system of type $A_2$. By abuse of notation we write \(\roots=A_2\). Let $\sroots=\{\alpha_1,\alpha_2\}$ and 
let \(r_0=\alpha_1,\,r_1=\alpha_1\rootp\alpha_2,\,r_2=\alpha_2,\,r_3=\rootm\alpha_1,\,r_4=\rootm\alpha_1\rootm\alpha_2,\,r_5=\rootm\alpha_2\).\\
Then
 \[
 T_1(A_2)=\{0_1,r_0,r_1,r_2,r_3,r_4,r_5\},
 \] 
\begin{align*}
T_2(A_2)=\{0_2,\,&\lbr r_0\br r_2\rbr,\lbr r_0\br r_3\rbr,\lbr r_0\br r_4\rbr,\lbr r_1\br r_3\rbr ,\lbr r_1\br r_4\rbr ,\lbr r_1\br r_5\rbr,\lbr r_2\br r_0\rbr,\lbr r_2\br r_4\rbr,\lbr r_2\br r_5\rbr,
\\&\lbr r_3\br r_0\rbr,\lbr r_3\br r_1\rbr,\lbr r_3\br r_5\rbr,\lbr r_4\br r_0\rbr,\lbr r_4\br r_1\rbr,\lbr r_4\br r_2\rbr,\lbr r_5\br r_1\rbr,\lbr  r_5\br r_2\rbr,\lbr  r_5\br r_3\rbr\}.
\end{align*}
\begin{figure}[h!]
\caption{The links $T_2(A_2)$ depicted by \colorchain\, arrows.}
\label{fig:C_2(A_2)}
\begin{center}
\begin{tikzpicture}[scale=\scaleAA]
  \path node at ( 0,0) [root,label=270: $ $] (zero) {$ $}	
	node at ( 4,0) [root,draw,label=0: $r_0{=}\alpha_1$] (one) {$ $}
  	node at ( 2,3.464) [root,draw,label=60: $r_1$] (two) {$ $}
  	node at ( -2,3.464) [root,draw,label=120: $r_2{=}\alpha_2$] (three) {$ $}
	node at ( -4,0) [root,draw,label=180: $r_3$] (four) {$ $}
	node at ( -2,-3.464) [root,draw,label=240: $r_4$] (five) {$ $}
	node at ( 2,-3.464) [root,draw,label=300: $r_5$] (six) {$ $};
	\draw[chain] (one.150) to node []{$ $} (four.30);
	\draw[chain] (four.330) to node []{$ $} (one.210);
	\draw[chain] (two.210) to node []{$ $} (five.90);
	\draw[chain] (five.30) to node []{$ $} (two.270);
	\draw[chain] (three.270) to node []{$ $} (six.150);
	\draw[chain] (six.90) to node []{$ $} (three.330);

	\draw[chain] (one.120) to node []{$ $} (three.0);
	\draw[chain] (three.300) to node [sloped,above]{$ $} (one.180);
	\draw[chain] (two.180) to node [sloped,above]{$ $} (four.60);
	\draw[chain] (four.0) to node [sloped,above]{$ $} (two.240);
	\draw[chain] (three.240) to node [sloped,above]{$ $} (five.120);
	\draw[chain] (five.60) to node [sloped,above]{$ $} (three.300);
	\draw[chain] (four.300) to node [sloped,above]{$ $} (six.180);
	\draw[chain] (six.120) to node [sloped,above]{$ $} (four.0);
	\draw[chain] (five.0) to node [sloped,above]{$ $} (one.240);
	\draw[chain] (one.180) to node [sloped,above]{$ $} (five.60);
	\draw[chain] (six.60) to node [sloped,above]{$ $} (two.300);
	\draw[chain] (two.240) to node [sloped,above]{$ $} (six.120);
\end{tikzpicture}
\end{center}
\end{figure}
\end{Example}

\subsection{Differential}\label{sec:differential}
\newcommand\dd{\mathsf{d}}
From here onwards, everything is written additively, to increase readability, in contrast to the multiplicative notation used in Section \ref{sec:intro}. One can then define \(\mathsf{\partial}^n:T_{n+1}(\roots)\rightarrow C_{n}(\roots)\) in the usual manner \cite{harrison1962commutative}.
Here we give the first instances, followed by the general formula:
\begin{align*}
\mathsf{\partial}^0 r_0&=0, \\
\mathsf{\partial}^1 \lbr r_0\br r_1\rbr&= r_1-\lbr r_0\rootp r_1\rbr+r_0,\\
\mathsf{\partial}^2 \lbr r_0\br r_1\br r_2\rbr&= \lbr r_1\br r_2\rbr-\lbr r_0\rootp r_1\br r_2\rbr+\lbr r_0\br r_1\rootp r_2\rbr-\lbr r_0\br r_1\rbr,\\
\mathsf{\partial}^{n} \lbr r_0\br \cdots\br r_n\rbr&=\lbr r_1\br \cdots\br r_n\rbr+\sum_{j=1}^n (-1)^{j} \lbr r_0\br \cdots\br r_{j-1}\rootp r_j\br \cdots\br r_n\rbr -(-1)^n \lbr r_0\br \cdots\br r_{n-1}\rbr.
\end{align*}
The boundary operator \(\mathsf{\partial}^{n}\) can be extended by linearity to a map \(\mathsf{\partial}^n:C_{n+1}(\roots)\rightarrow C_{n}(\roots)\) and thus, since it follows from the usual computation that \(\mathsf{\partial}^{n}\mathsf{\partial}^{n+1}=0\), one obtains a chain complex.

The coboundary \(\mathsf{d}^n:C^n(\roots,M)\rightarrow C^{n+1}(\roots,M)\) is defined 
by \[\mathsf{d}^n \omega^n(r_0 \ibr \cdots \ibr r_n)=\omega^n(\mathsf{\partial}^{n} \lbr r_0\br \cdots\br r_n\rbr)\] or, explicitly on the generators,
\begin{align*}
\mathsf{d}^0 \omega^0(r_0)&=0,\\
\mathsf{d}^1 \omega^1(r_0 \ibr r_1)&= \omega^1(r_1)-\omega^1(r_0\rootp r_1)+\omega^1(r_0),\\
\mathsf{d}^2 \omega^2(r_0 \ibr r_1 \ibr r_2)&= \omega^2(r_1 \ibr r_2)-\omega^2(r_0\rootp r_1 \ibr r_2)+\omega^2(r_0 \ibr r_1\rootp r_2)-\omega^2(r_0 \ibr r_1),
\end{align*}
\begin{align*}
\mathsf{d}^n \omega^n(r_0 \ibr \cdots \ibr r_n)&=\omega^n(r_1 \ibr \cdots \ibr r_n)\\
&\phantom{=}+\sum_{j=1}^n (-1)^{j} \omega^n(r_0 \ibr \cdots \ibr r_{j-1}\rootp r_j \ibr \cdots \ibr r_n)-(-1)^n\omega^n(r_0 \ibr \cdots \ibr r_{n-1}).
\end{align*}
That \(\mathsf{d}^{n+1}\mathsf{d}^n=0\) follows from the definition and the fact that \(\partial^n\) is a boundary operator.
We also want \(\mathsf{\partial}^{n} 0_{n+1}=0_{n}\). Let us check that this is indeed the case:
if \(r_0=0\), then 
\[
\partial^n \lbr 0\br r_1\br \cdots\br r_n\rbr=\lbr r_1\br\cdots\br r_i\rbr-\lbr 0\rootp r_1\br \cdots\br r_n\rbr=0_{n}\,.
\]
If \(r_k=0, k=1,\ldots,n-1\) then
\begin{multline*}
\partial^n \lbr r_0\br \cdots\br r_{k-1}\br 0\br r_{k+1}\br \cdots\br r_n\rbr
\\=(-1)^{k+1}\lbr r_0\br \cdots\br r_{k-1}\rootp 0\br \cdots\br r_n\rbr+(-1)^k \lbr r_0\br \cdots\br 0\rootp r_{k+1}\br \cdots\br r_n\rbr=0_{n}
\end{multline*}
and if \(r_n=0\), then 
\[
\partial^n \lbr r_0\br \cdots\br r_{n-1}\br 0\rbr=(-1)^n \lbr r_0\br \cdots\br r_{n-1}\rootp 0\rbr-(-1)^n \lbr r_0\br \cdots\br r_{n-1}\rbr=0_{n}\,.
\]
In the following sections we consider a number of concepts, namely cup product (Section \ref{sec:cupproduct}), reversal symmetry (Section \ref{sec:reversalsymmetry}) and equivariance with respect to automorphisms of the root system (Section \ref{sec:equivariance}),
before going to our main result, the acyclicity of the second cohomology in Section \ref{sec:acyclicity}.
These sections can all be read independently, with the exception of Section \ref{sec:revcup}, as it depends on Section \ref{sec:cupproduct} and of Section \ref{sec:acyclicity}, as it depends on Section \ref{sec:reversalsymmetry}.

\section{Cup product}\label{sec:cupproduct}
The cup product provides a strong motivation for cohomology theory, compared to homology theory, since its product structure turns a cohomology space into  a graded ring. In this section we define the cup product in the context of root systems cohomology, giving $C^\bullet(\roots,M)$ a ring structure. 
\begin{Definition}
The cup product of two forms is defined by
\[
\omega^p\cup\omega^q(r_1 \ibr \cdots \ibr r_{p+q})=\omega^p(r_1 \ibr \cdots \ibr r_p)\omega^q(r_{p+1} \ibr \cdots \ibr r_{p+q}),\quad \lbr r_1\br\cdots\br r_{p+q}\rbr\in T_{p+q}(\roots)\,.
\]
\end{Definition}
\begin{Remark}
The multiplication here is the multiplication as defined in Definition \ref{def:Mon} (or, more abstractly, if \(M=R^k\), the multiplication in \(R\)). One has to multiply the powers of \(\mb{I}\) to get the right result: with \(\omega^1(r_0)=1\) and \(\omega^1(r_1)=\mb{I}\),
   then \(\omega^1\cup\omega^1(r_0 \ibr r_1)=1\), if \(\lbr r_0\br r_1\rbr \in T_2(\roots)\),
whereas \(\omega^1(r_0)+\omega^1(r_1)=\mb{I}\) (the powers are added here).
\end{Remark}
It is clear that \(\lbr r_1\br\cdots\br r_{p}\rbr\in T_{p}(\roots)\) and \(\lbr r_{p+1}\br\cdots\br r_{p+q}\rbr\in T_{q}(\roots)\) if \(\lbr r_1\br\cdots\br r_{p+q}\rbr\in T_{p+q}(\roots)\), since there are less sums to check for existence
in each of the partial terms.
\begin{Lemma}
The following product rule holds:
\[
\dd^{p+q}\omega^p\cup\omega^q
=\dd^p\omega^p\cup\omega^q
+(-1)^{p} \omega^p\cup\dd^q \omega^q.
\]
\end{Lemma}
\begin{proof}
\begin{align*}
&\dd^{p+q}\omega^p\cup\omega^q(r_0 \ibr \cdots \ibr r_p \ibr r_{p+1} \ibr \cdots \ibr r_{p+q})
\\&= \omega^p(r_1 \ibr \cdots \ibr r_p)\omega^q(r_{p+1} \ibr \cdots \ibr r_{p+q})
-(-1)^{p+q}\omega^p(r_0 \ibr \cdots \ibr r_{p-1})\omega^q(r_{p} \ibr \cdots \ibr r_{p+q-1})
\\&\phantom{=}+\sum_{j=1}^p (-1)^{j} \omega^p(r_0 \ibr \cdots \ibr r_{j-1}\rootp r_j \ibr \cdots \ibr r_p)\omega^q(r_{p+1} \ibr \cdots \ibr r_{p+q})
\\&\phantom{=}+\sum_{j=p+1}^{p+q} (-1)^{j} \omega^p(r_0 \ibr \cdots \ibr r_{p-1})\omega^q(r_p \ibr \cdots \ibr r_{j-1}\rootp r_j \ibr \cdots \ibr r_{p+q})
\\&=
(-1)^{p}\omega^p(r_0 \ibr \cdots \ibr r_{p-1})\Big[-(-1)^q \omega^q(r_{p} \ibr \cdots \ibr r_{p+q-1})
\\&\phantom{=}+ \sum_{j=1}^{q} (-1)^{j} \omega^q(r_p \ibr \cdots \ibr r_{j+p-1}\rootp r_{j+p} \ibr \cdots \ibr r_{p+q})+\omega^q(r_{p+1} \ibr \cdots \ibr r_{p+q}) \Big]
\\&\phantom{=}+\Big[ \omega^p(r_1 \ibr \cdots \ibr r_p) + \sum_{j=1}^p (-1)^{j} \omega^p(r_0 \ibr \cdots \ibr r_{j-1}\rootp r_j \ibr \cdots \ibr r_p)
\\&\phantom{=}-  (-1)^p \omega^p(r_0 \ibr \cdots \ibr r_{p-1})\Big] 
\omega^q(r_{p+1} \ibr \cdots \ibr r_{p+q})
\\&=\dd^p\omega^p(r_0 \ibr \cdots \ibr r_p)\omega^q(r_{p+1} \ibr \cdots \ibr r_{p+q})
+(-1)^{p} \omega^p(r_0 \ibr \cdots \ibr r_{p-1})\dd^q \omega^q(r_p \ibr \cdots \ibr r_{p+q}).
\\&=\dd^p\omega^p\cup\omega^q(r_{0} \ibr \cdots \ibr r_{p+q})
+(-1)^{p} \omega^p\cup\dd^q \omega^q(r_0 \ibr \cdots \ibr r_{p+q}).
\end{align*}
\end{proof}
\begin{Corollary}
The closed forms form a subring of $C^{\text{\textbullet}}(\roots,M)$ and the exact forms form an ideal within that subring.
\end{Corollary}


\section{Reversal symmetry}\label{sec:reversalsymmetry}
Notice that when  \(\lbr r_1 \br r_2\br\cdots\br r_n\rbr\) is in \(T_n(\roots)\) then its opposite
\[
 \rho^n\lbr r_1 \br r_2\br\cdots\br r_n\rbr=\lbr r_n \br r_{n-1}\br\cdots\br r_1\rbr\,,
 \] 
 is in  \(T_n(\roots)\) as well, so \(\rho^n\) is an involution on \(T_n(\roots)\).
We can in fact make a stronger statement.
The symmetric group $S_n$ acts on the product $\roots_0\times\ldots\times\roots_0$ by permuting the $n$ factors.
To the subset $T_n(\roots)$ of this product is associated a stabiliser subgroup of $S_n$ which turns out to be generated by the reversal permutation $\rho^n$. We begin by proving this fact (Lemma \ref{lem:stabiliser generated by reversal}) and subsequently we use the reversal permutation to split the cochain complex into a symmetric and antisymmetric part. It is the symmetric part that is related to ALiAs.

We start by describing a few links, using a series of subsets $U_n$, $V_n$, $W_n$ of $(\roots_0)^n$, $n\ge 1$. 
Let 
$U_n$ consist of all elements
\[u_{n,j}=[\alpha | \rootm \alpha | \cdots| \rootpm\alpha|\rootpm\beta|\rootmp(\alpha\rootp\beta)|\rootpm\alpha|\cdots|  (\rootm 1)^{n} \alpha]\]
where $\alpha,\beta,\alpha\rootp\beta\in\roots$. The subscript $j\in\{1,\ldots,n\}$ denotes the position of $\rootpm\beta$.
Let $V_n$ consist of all elements
\[v_{n,j}=[\alpha | \rootm \alpha | \cdots| \rootpm\alpha|\rootpm\beta|\rootmp\beta|\rootmp\alpha|\cdots|  (\rootm 1)^{n-1} \alpha]\]
where again $\alpha,\beta,\alpha\rootp\beta\in\roots$ and $j\in\{1,\ldots,n\}$ denotes the position of $\rootpm\beta$.
Let $W_n$ consist of all elements
\[w_{n}=[\alpha | \rootm \alpha |\cdots|  (\rootm 1)^{n-1} \alpha]\]
where $\alpha\in\roots$.
Define \[X_n=U_{n}\cup V_{n}\cup W_{n}\cup\{0_{n}\}.\]
\begin{Lemma}
\label{lem:chains a b c}
$X_n\subset T_n(\roots)$.
\end{Lemma}
\begin{proof}
For $i=1,\ldots,n-1$, define \[f_i([r_1|\ldots|r_n])=[r_1|\ldots|r_{i-1}|r_i\rootp r_{i+1}|r_{i+2}|\ldots|r_{n}]\] just to give a name to the operation used in the definition of $T_n(\roots)$.
We find that
\begin{align*}
f_i (a_{n,j})&=0_{n-1},\quad i=1\ldots,j-2,j+2,\ldots,n-1,
\\f_{j-1} (a_{n,j})&\in V_{n-1},
\quad f_{j} (a_{n,j})\in W_{n-1},
\quad f_{j+1} (a_{n,j})\in V_{n-1},
\\[2mm]
f_i (b_{n,j})&=0_{n-1},\quad i=1\ldots,j-2,j,j+2,\ldots,n-1,
\\f_{j-1} (b_{n,j})&\in U_{n-1},
\quad f_{j+1} (b_{n,j})\in U_{n-1},
\\[2mm]
f_i (c_{n,j})&=0_{n-1},\quad i=1\ldots n-1.
\end{align*}
For $n=1$ the statement of the lemma is trivial. The general statement follows by induction, because $f_i(X_n)\subset  X_{n-1}$ for $i=1,\ldots,n-1$.
\end{proof}

\begin{Lemma}
\label{lem:stabiliser generated by reversal}
Let a permutation $\pi\in S_n$ act on a tuple \([r_1 | r_2|\cdots| r_n]\) by permuting the factors: \[\pi\cdot [r_1 | r_2|\cdots| r_n]=[r_{\pi^{-1}1} | r_{\pi^{-1}2}|\cdots| r_{\pi^{-1}n}].\] The subgroup of $S_n$ preserving $T_n(\roots)$ consists only of the identity permutation $e^n$ and the reversal permutation $\rho^n:i\mapsto n-(i-1)$.
\end{Lemma}
\begin{proof}
First we argue that the permutations $\{e^n,\rho^n\}$ are precisely those permutations $\pi$ such that 
\begin{equation}
\label{eq:reversal symmetry condition}
|\pi i - \pi (i+1)|=1,\quad 1\le i \le n-1.
\end{equation}
It is clear that $e^n$ and $\rho^n$ satisfy (\ref{eq:reversal symmetry condition}). Now let $\pi$ be a permutation satisfying  (\ref{eq:reversal symmetry condition}). Let $\pi 1 =j$. Then $\pi 2 \in\{j+1,j-1\}$. If $\pi 2=j+1$, then $\pi 3=j+2$ and $\pi i=j+{i-1}$. It follows that $j=1$ and $\pi=e^n$. If on the other hand $\pi 2=j-1$, then $\pi i=j-(i-1)$ and it follows that $j=n$ and $\pi=\rho^n$. 
 
We turn to the stabiliser of $T_n(\roots)$ in $S_n$. Notice first that $e^nT_n(\roots)=T_n(\roots)$ and $\rho^nT_n(\roots)=T_n(\roots)$. Now take another permutation $\pi\in S_n$. Then there is $1\le i\le n-1$ such that $|\pi i - \pi (i+1)|>1$. 
Suppose that $|\pi i - \pi (i+1)|$ is even. 
Then $\pi^{-1} c_n$ has at positions $i$ and $i+1$ two factors of $c_n$ at even distance of each other. Given that $c_n$ has alternating factors $\alpha$ and $\rootm\alpha$, we see that $f_i\pi^{-1} c_n$ has a factor $\alpha\rootp\alpha$ or $\rootm\alpha\rootm\alpha$ at position $i$. This factor is therefore not in $\roots_0$ hence $\pi^{-1} c_n\notin T_n(\roots)$. But $c_n\in T_n(\roots)$ by Lemma $\ref{lem:chains a b c}$ so we conclude that $\pi^{-1}$, and hence $\pi$, is not an element of $(S_n)_{T_n(\roots)}$.

Suppose now that $|\pi i - \pi (i+1)|$ is odd and greater than $1$. Choose $j=\min\{\pi i, \pi (i+1)\}+1$. Then $\pi^{-1} a_{n,j}$ has at position $i$ and $i+1$ factors on different sides of the two factors in $a_{n,j}$ involving $\beta$, and at odd distance from each other. Therefore, $f_i\pi^{-1} a_{n,j}$ has a factor $\alpha\rootp\alpha$ or $\rootm\alpha\rootm\alpha$ at position $i$.
Now the same reasoning as above shows that $\pi\notin (S_n)_{T_n(\roots)}$.
\end{proof}
We turn to the interplay between the cochain complex and the reversal permutation.
Let \(\rho^n \omega^n(\phi) =\omega^n(\rho^n(\phi))\). 
\begin{Definition}
Let \(\kappa_n={n+1\choose 2}+1\).
Let \(\hat{\rho}^n=(-1)^{\kappa_n}\rho^n\).
We say that \(\omega^n\) is \emph{symmetric} if \(\hat{\rho}^n\omega^n=\omega^n\) and \emph{antisymmetric} if \(\hat{\rho}^n\omega^n=-\omega^n\) (cf. Remark \ref{rem:antisymmetry}).
Let \(C_{\pm}^n(\roots,M)\) consist of those \(\omega^n\in C^n(\roots,M)\) such that \(\hat{\rho}^n \omega^n=\pm\omega^n\).
\end{Definition}
\begin{Lemma}[{\cite[Section 3]{gerstenhaber1987hodge}}]
\label{lem:reversal symmetry}
One has \(\hat{\rho}^{n+1}\mathsf{d}^n=\mathsf{d}^n \hat{\rho}^n \). 
\end{Lemma}
\begin{Example}
\begin{align*}
\mathsf{d}^1 \omega^1(r_0 \ibr r_1)&= \omega^1(r_1)-\omega^1(r_0\rootp r_1)+\omega^1(r_0)\\
\hat{\rho}^2\mathsf{d}^1 \omega^1(r_0 \ibr r_1)&= \omega^1(r_1)-\omega^1(r_0\rootp r_1)+\omega^1(r_0)
\\&=
\hat{\rho}^1(\omega^1(r_1)-\omega^1(r_0\rootp r_1)+\omega^1(r_0))=\mathsf{d}^1 
\hat{\rho}^1 \omega^1(r_0 \ibr r_1)\\
\mathsf{d}^2 \omega^2(r_0 \ibr r_1 \ibr r_2)&= \omega^2(r_1 \ibr r_2)-\omega^2(r_0\rootp r_1 \ibr r_2)+\omega^2(r_0 \ibr r_1\rootp r_2)-\omega^2(r_0 \ibr r_1)\\
\hat{\rho}^3\mathsf{d}^2 \omega^2(r_0 \ibr r_1 \ibr r_2)&=-( \omega^2(r_1 \ibr r_0)-\omega^2(r_2\rootp r_1 \ibr r_0)+\omega^2(r_2 \ibr r_1\rootp r_0)-\omega^2(r_2 \ibr r_1))
\\&=\hat{\rho}^2 (-\omega^2(r_0 \ibr r_1)+\omega^2(r_0 \ibr r_2\rootp r_1)-\omega^2(r_1\rootp r_0 \ibr r_2)+\omega^2(r_1 \ibr r_2))
\\&=\mathsf{d}^2 \hat{\rho}^2\omega^2(r_0 \ibr r_1 \ibr r_2)
\end{align*}
\end{Example}
\begin{proof}
Notice that \(\kappa_n +\kappa_{n+1}=(n+1)^2+2\equiv n+1 \mod 2\).
\begin{align*}
\hat{\rho}^{n+1}\mathsf{d}^n \omega^n(r_0 \ibr \cdots \ibr r_n)&=(-1)^{\kappa_{n+1}}\Big[\omega^n(r_{n-1} \ibr \cdots \ibr r_0)-(-1)^n\omega^n(r_{n} \ibr \cdots \ibr r_{1})
\\&\phantom{=}+\sum_{j=1}^n (-1)^{j} \omega^n(r_n \ibr \cdots \ibr r_{n-j+2} \ibr r_{n-j+1}\rootp r_{n-j} \ibr r_{n-j-1} \ibr \cdots \ibr r_0)\Big]
\\&=(-1)^{\kappa_{n+1}+\kappa_{n}}\hat{\rho}^n\Big[\omega^n(r_0 \ibr \cdots \ibr r_{n-1})-(-1)^n\omega^n(r_{1} \ibr \cdots \ibr r_{n})
\\&\phantom{=}+\sum_{j=1}^n (-1)^{j} \omega^n(r_0 \ibr \cdots \ibr r_{n-j}\rootp  r_{n-j+1} \ibr \cdots \ibr r_n)\Big]
\\&=(-1)^{n+1}\hat{\rho}^n\Big[\omega^n(r_0 \ibr \cdots \ibr r_{n-1})-(-1)^n\omega^n(r_{1} \ibr \cdots \ibr r_{n})
\\&\phantom{=}+ \sum_{j=1}^n (-1)^{n-j+1} \omega^n(r_0 \ibr \cdots \ibr r_{j}\rootp  r_{j-1} \ibr \cdots \ibr r_n)\Big]
\\&=(-1)^{n+1}(-1)^{n+1}\hat{\rho}^n\big[-(-1)^{n}\omega^n(r_0 \ibr \cdots \ibr r_{n-1})+\omega^n(r_{1} \ibr \cdots \ibr r_{n})
\\&\phantom{=}+ \sum_{j=1}^n (-1)^{j} \omega^n(r_0 \ibr \cdots \ibr r_{j-1}\rootp r_{j} \ibr \cdots \ibr r_n)\Big]
\\&=\mathsf{d}^n \hat{\rho}^n \omega^n (r_0 \ibr \cdots \ibr r_n).
\end{align*}
This proves the statement.
\end{proof}
\begin{Corollary}
\(\mathsf{d}^n\) maps \(C_{\pm}^n(\roots,M)\) to \(C_{\pm}^{n+1}(\roots,M)\).
\end{Corollary}
Notice in particular that \(C^1(\roots,M)=C^1_+(\roots,M)\) and \(H^2_-(\roots,M)=Z^2_-(\roots,M)\).
\begin{Example}[$H^2_-(A_2,M)\ne 0$]
\label{ex:H^2_-ne0}
Let $\roots$ be the root system of type $A_2$. As in Example \ref{ex:chains on A2} we number the roots \(r_0=\alpha_1,\,r_1=\alpha_1\rootp\alpha_2,\,r_2=\alpha_2,\,r_3=\rootm\alpha_1,\,r_4=\rootm\alpha_1\rootm\alpha_2,\,r_5=\rootm\alpha_2\) and consider the indices of $r$ as elements in $\Zn{6}$. Notice that $r_i\rootp r_{i+2}= r_{i+1}$, $r_i\rootp r_{i+3}=0$ and $r_i\rootp r_{i+4}= r_{i+5}$.
The $3$-links are the triples $\lbr\alpha\br\beta\br\gamma\rbr$ with the property that $ \alpha\rootp\beta \in\roots_0$, $\beta\rootp\gamma\in\roots_0$ and  $\alpha\rootp\beta\rootp\gamma\in\roots_0$.
The complete set \(T_3(A_2)\) is given by 
\begin{align*}
a_i&=\lbr r_i\br r_{i+2}\br r_{i+4}\rbr,\\
b_i&=\lbr r_i\br r_{i+2}\br r_{i+5}\rbr,\\
c_i&=\lbr r_i\br r_{i+3}\br r_{i+5}\rbr,\\
d_i&=\lbr r_i\br r_{i+3}\br r_{i+1}\rbr=-\hat{\rho}^3 b_i,\\
e_i&=\lbr r_i\br r_{i+4}\br r_{i+1}\rbr=-\hat{\rho}^3 c_i,\\
f_i&=\lbr r_i\br r_{i+4}\br r_{i+2}\rbr=-\hat{\rho}^3 a_i,\quad i\in\Zn{6}.
\end{align*}
Let us define an antisymmetric $2$-cochain $\omega^2_-\in C^2_-(A_2,M)$ by 
\[\omega^2_-(r_i \ibr r_{i+2})=1, \quad\omega^2_-(r_i \ibr r_{i+3})=0,\quad i\in\Zn{6}.\]
\begin{figure}[H]
\caption{An antisymmetric $2$-cocycle on $A_2$. Links sent to values $1$, $-1$ and $0$ are depicted in \colorp, \colorm\;and \colorchain\;respectively. }
\begin{center}
\begin{tikzpicture}[scale=\scaleAA]
  \path node at ( 0,0) [root,label=270: $ $] (zero) {$ $}	
	node at ( 4,0) [root,draw,label=0: $r_0{=}\alpha_1$] (one) {$ $}
  	node at ( 2,3.464) [root,draw,label=60: $r_1$] (two) {$ $}
  	node at ( -2,3.464) [root,draw,label=120: $r_2{=}\alpha_2$] (three) {$ $}
	node at ( -4,0) [root,draw,label=180: $r_3$] (four) {$ $}
	node at ( -2,-3.464) [root,draw,label=240: $r_4$] (five) {$ $}
	node at ( 2,-3.464) [root,draw,label=300: $r_5$] (six) {$ $};

	\draw[chain] (one.150) to node []{$ $} (four.30);
	\draw[chain] (four.330) to node []{$ $} (one.210);
	\draw[chain] (two.210) to node []{$ $} (five.90);
	\draw[chain] (five.30) to node []{$ $} (two.270);
	\draw[chain] (three.270) to node []{$ $} (six.150);
	\draw[chain] (six.90) to node []{$ $} (three.330);

	\draw[cochainp] (one.120) to node []{$ $} (three.0);
	\draw[cochainm] (three.300) to node [sloped,above]{$ $} (one.180);
	\draw[cochainp] (two.180) to node [sloped,above]{$ $} (four.60);
	\draw[cochainm] (four.0) to node [sloped,above]{$ $} (two.240);
	\draw[cochainp] (three.240) to node [sloped,above]{$ $} (five.120);
	\draw[cochainm] (five.60) to node [sloped,above]{$ $} (three.300);
	\draw[cochainp] (four.300) to node [sloped,above]{$ $} (six.180);
	\draw[cochainm] (six.120) to node [sloped,above]{$ $} (four.0);
	\draw[cochainp] (five.0) to node [sloped,above]{$ $} (one.240);
	\draw[cochainm] (one.180) to node [sloped,above]{$ $} (five.60);
	\draw[cochainp] (six.60) to node [sloped,above]{$ $} (two.300);
	\draw[cochainm] (two.240) to node [sloped,above]{$ $} (six.120);
\end{tikzpicture}
\end{center}
\end{figure}
Then $\omega^2_-\in Z^2_-(A_2,M)$. Indeed,
\begin{align*}
\mathsf{d}^2\omega^2_-(a_i)&=\mathsf{d}^2\omega^2_-(r_i \ibr  r_{i+2} \ibr  r_{i+4})\\
&=\omega^2_-(r_{i+2} \ibr  r_{i+4})-\omega^2_-(r_{i+1} \ibr  r_{i+4})+\omega^2_-(r_{i} \ibr  r_{i+3})-\omega^2_-(r_{i} \ibr  r_{i+2})\\
&=1-0+0-1=0,
\end{align*}
\begin{align*}
\mathsf{d}^2\omega^2_-(b_i)&=\mathsf{d}^2\omega^2_-(r_i \ibr  r_{i+2} \ibr  r_{i+5})\\
&=\omega^2_-(r_{i+2} \ibr  r_{i+5})-\omega^2_-(r_{i+1} \ibr  r_{i+5})+\omega^2_-(r_{i} \ibr  0)-\omega^2_-(r_{i} \ibr  r_{i+2})\\
&=0-(-1)+0-1=0,
\end{align*}
\begin{align*}
\mathsf{d}^2\omega^2_-(c_i)&=\mathsf{d}^2\omega^2_-(r_i \ibr  r_{i+3} \ibr  r_{i+5})\\
&=\omega^2_-(r_{i+3} \ibr  r_{i+5})-\omega^2_-(0 \ibr  r_{i+5})+\omega^2_-(r_{i} \ibr  r_{i+4})-\omega^2_-(r_{i} \ibr  r_{i+3})\\
&=1-0-1-0=0.
\end{align*}
That $\mathsf{d}^2\omega^2_-$ vanishes at the remaining links now follows from Lemma \ref{lem:reversal symmetry}, e.g. 
\[
\mathsf{d}^2\omega^2_-(d_i)=\mathsf{d}^2\omega^2_-(-\hat{\rho}^3 b_i)=-\mathsf{d}^2\hat{\rho}^3\omega^2_-(b_i)=-\hat{\rho}^2\mathsf{d}^2\omega^2_-(b_i)=0.
\] 
This example shows the non obvious fact that $Z_-^2(A_2,M)=H_-^2(A_2,M)\ne 0$.
\end{Example}

For our purpose it is sufficient to work with \(C_{+}^{\text{\textbullet}}(\roots,M)\).
Notice that the usual splitting into symmetric and antisymmetric is not immediately applicable here, since it involves division by \(2\),
which means taking square roots in the multiplicative context (see Example \ref{ex:B2}).

\subsection{Reversal symmetry and cup product}\label{sec:revcup}
\begin{Lemma}
\(\hat{\rho}^{p+q}\omega^{p}\cup\omega^q=(-1)^{\kappa_{p+q}+\kappa_{p}+\kappa_{q}}\hat{\rho}^q\omega^q\cup\hat{\rho}^p\omega^p\).
\end{Lemma}
\begin{proof}
\begin{align*}
\hat{\rho}^{p+q}\omega^{p}\cup\omega^q(r_1 \ibr \cdots \ibr r_{p+q})&=
(-1)^{\kappa_{p+q}}\omega^{p+q}(r_{p+q} \ibr \cdots \ibr r_1)
\\&=(-1)^{\kappa_{p+q}}\omega^p(r_{p+q} \ibr \cdots \ibr r_{q+1})\omega^q(r_q \ibr \cdots \ibr r_1)
\\&=(-1)^{\kappa_{p+q}+\kappa_{p}+\kappa_{q}}\hat{\rho}^p\omega^p(r_{q+1} \ibr \cdots \ibr r_{p+q})\hat{\rho}^q\omega^q(r_1 \ibr \cdots \ibr r_q)
\\&=(-1)^{\kappa_{p+q}+\kappa_{p}+\kappa_{q}}\hat{\rho}^q\omega^q\cup\hat{\rho}^p\omega^p(r_1 \ibr \cdots \ibr r_{p+q})
\end{align*}
\end{proof}
It follows then 
\begin{Lemma}
\(\hat{\rho}^{p+q}\omega^{p}\cup\omega^q=(-1)^{pq+1}\hat{\rho}^q\omega^q\cup\hat{\rho}^p\omega^p\).
\end{Lemma}
\begin{proof}
The proof follows by noticing that  \({\kappa_{p+q}+\kappa_{p}+\kappa_{q}}\equiv {pq+1}\mod{2}\).
\end{proof}
\section{Symmetries of the root system}\label{sec:equivariance}
In this section we show that the (co)homology is equivariant with respect to automorphisms of the root system. 

Let $\sigma\in\Aut(\roots)$.
Then \(\sigma(\beta\rootp\gamma)=\sigma(\beta)\rootp\sigma(\gamma)\).
We define 
\(\sigma \lbr r_1\br \cdots\br r_n\rbr=\lbr\sigma r_1\br \cdots\br \sigma r_n\rbr\), and by linear extension this defines an action on the chains $C_{\text{\textbullet}}(\roots)$.
\begin{Lemma}
The action of $\Aut(\roots)$ on $C_{\text{\textbullet}}(\roots)$ commutes with the differential \(\mathsf{\partial}\).
\end{Lemma}
\begin{proof}
By a straightforward computation
\(\sigma \mathsf{\partial}^n \lbr r_0\br \cdots\br r_n\rbr =\mathsf{\partial}^n \sigma  \lbr r_0\br \cdots\br r_n\rbr\).
\end{proof}
\begin{Lemma}
The action of $\Aut(\roots)$ on the cochains $C^{\text{\textbullet}}(\roots,M)$ commutes with the differential \(\mathsf{d}\).
\end{Lemma}
\begin{proof} Again a straightforward computation
\begin{align*}
\sigma \mathsf{d}^n \omega^n(r_0 \ibr \cdots \ibr r_n)&=\mathsf{d}^n \omega^n(\sigma^{-1}(r_0 \ibr \cdots \ibr r_n))\\
&=\omega^n( \mathsf{\partial}^n \sigma^{-1}(r_0 \ibr \cdots \ibr r_n))\\
&=\omega^n(\sigma^{-1} \mathsf{\partial}^n(r_0 \ibr \cdots \ibr r_n))\\
&=\sigma\omega^n(\mathsf{\partial}^n(r_0 \ibr \cdots \ibr r_n))\\
&=\mathsf{d}^n \sigma\omega^n(r_0 \ibr \cdots \ibr r_n)
\end{align*}
proves the statement.
\end{proof}


\section{The second cohomology group $H^2_+(\roots,M)$}\label{sec:acyclicity}
\subsection{Acyclicity}
In this section we show that all symmetric $2$-cocycles are integrable.
\begin{Lemma}\label{lem:properties of roots} Consider a reduced root system $\roots$ with positive roots $\roots^+$ and base $\sroots$. 
Addition in the ambient Euclidean space is denoted by $\rootp$ and $(\cdot\,,\cdot)$ is the invariant inner product.
\begin{enumerate}
\item \label{item:obtuse angle adds to root}
If the angle between two nonproportional roots is strictly obtuse, then their sum is a root.
\item \label{item:simple roots obtuse}
If $\alpha,\alpha'\in\sroots$, $\alpha\ne\alpha'$, then $\inp{\alpha}{\alpha'}\le 0$ and $\alpha-\alpha' \notin \roots$.
\item \label{item:positive root has sharp angle with simple root}
If $\beta\in\roots^+\setminus\sroots$, then there exists $\alpha\in\sroots$ such that $\inp{\beta}{\alpha}>0$. In particular $\beta\rootm\alpha\in\roots$.
\item \label{item:minus two simple roots}
Let $\alpha_i,\;\alpha_j$ be distinct simple roots. If $\beta$ is a root, and $\gamma_i=\beta\rootm\alpha_i$ and $\gamma_j=\beta\rootm\alpha_j$ are roots or zero, then $\delta=\beta\rootm\alpha_i\rootm\alpha_j$ is a root.
\item \label{item:subtract simple root from sum and summand}
If two roots $\beta$ and $\beta'$ add up to a positive root that is not simple, then there exists a simple root $\alpha$ such that $\beta\rootp\beta'\rootm\alpha$ is a root and either $\beta\rootm\alpha$ or $\beta'\rootm\alpha$ is a root as well.
\end{enumerate}
\end{Lemma}
\begin{proof}
Statements \ref{item:obtuse angle adds to root}, \ref{item:simple roots obtuse} and \ref{item:positive root has sharp angle with simple root} are well known and can be found in standard texts covering root systems, e.g.~\cite{MR0323842, MR1890629}. 
Statement \ref{item:minus two simple roots} can be proved as follows. Since $\beta\ne 0$ we have
\begin{align*}0&<\inp{\beta}{\beta}=\inp{\alpha_i\rootp\gamma_i}{\alpha_j\rootp\gamma_j}=\inp{\alpha_i}{\alpha_j}+\inp{\alpha_i}{\gamma_j}+\inp{\gamma_i}{\alpha_j}+\inp{\gamma_i}{\gamma_j}.
\end{align*}
From statements \ref{item:obtuse angle adds to root} and \ref{item:simple roots obtuse} we know that $\inp{\alpha_i}{\alpha_j}\le 0$ and $\inp{\gamma_i}{\gamma_j}\le 0$ since $\gamma_i\rootm\gamma_j=\alpha_j\rootm\alpha_i$ is not a root. 
It follows that
\[\inp{\alpha_i}{\gamma_j}+\inp{\gamma_i}{\alpha_j}>0.\] Therefore at least one of these terms is positive and at least one of $\gamma_i\rootm\alpha_j=\delta$ and $\gamma_j\rootm\alpha_i=\delta$ is a root (by  statement \ref{item:obtuse angle adds to root}) as desired.

Finally, to prove statement \ref{item:subtract simple root from sum and summand} one observes that 
by statement \ref{item:positive root has sharp angle with simple root} there is a simple root $\alpha$ such that $0<\inp{\beta\rootp\beta'}{\alpha}=\inp{\beta}{\alpha}+\inp{\beta'}{\alpha}$. Hence at least one of the inner products on the right hand side is positive and therefore the involved roots can be subtracted in $\roots$, by statement \ref{item:obtuse angle adds to root}.
\end{proof}
\begin{Lemma}
If \(\omega^2\in Z^2(\roots,M)\) then 
\begin{align}
\label{eq:symmetry of killing form}
\omega^2(\alpha \ibr \rootm\alpha)&=\omega^2(\rootm\alpha \ibr \alpha),\\
\label{eq:identity 1}
\omega^2(\alpha \ibr \beta)&=
\omega^2(\alpha \ibr \rootm\alpha)-\omega^2(\rootm\alpha \ibr \alpha\rootp\beta)\\
&=\omega^2(\beta \ibr \rootm\beta)-\omega^2(\alpha\rootp\beta \ibr \rootm\beta),
\label{eq:identity 2}
\end{align}
for all $\lbr\alpha\br\beta\rbr\in T_2(\roots)$.
\end{Lemma}
\begin{proof}
If $\lbr\alpha\br\beta\rbr\in T_2(\roots)$ then 
\[
\lbr\alpha\br\rootm\alpha\br\alpha\rbr,\quad
\lbr\alpha\br\rootm\alpha\br\alpha\rootp\beta\rbr,\quad
\lbr\alpha\br\beta\br\rootm\beta\rbr,
\]
are in $T_3(\roots)$.
Equations (\ref{eq:symmetry of killing form})-(\ref{eq:identity 2}) are rearrangements of $\mathsf{d}^2\omega^2$ evaluated in these $3$-links respectively:
\begin{align*}
0&=\mathsf{d}^2\omega^2(\alpha \ibr \rootm\alpha \ibr \alpha)\\
&=\omega^2(\rootm\alpha \ibr \alpha)
-\omega^2(0 \ibr \alpha)
+\omega^2(\alpha \ibr 0)-\omega^2(\alpha \ibr \rootm\alpha)\\
&=\omega^2(\rootm\alpha \ibr \alpha)
-\omega^2(\alpha \ibr \rootm\alpha),\\
0&=\mathsf{d}^2\omega^2(\alpha \ibr \rootm\alpha \ibr \alpha\rootp\beta)\\
&=\omega^2(\rootm\alpha \ibr \alpha\rootp\beta)-\omega^2(0 \ibr \alpha\rootp\beta)
+\omega^2(\alpha \ibr \beta)-\omega^2(\alpha \ibr \rootm\alpha)\\
&=\omega^2(\rootm\alpha \ibr \alpha\rootp\beta)+\omega^2(\alpha \ibr \beta)-\omega^2(\alpha \ibr \rootm\alpha),\\
0&=\mathsf{d}^2\omega^2(\alpha \ibr  \beta \ibr  \rootm\beta)\\
&=\omega^2(\beta \ibr \rootm\beta)-\omega^2(\alpha\rootp\beta \ibr \rootm\beta)
+\omega^2(\alpha \ibr  0)-\omega^2(\alpha \ibr  \beta)\\
&=\omega^2(\beta \ibr \rootm\beta)-\omega^2(\alpha\rootp\beta \ibr \rootm\beta)
-\omega^2(\alpha \ibr \beta).
\end{align*}
\end{proof}
Notice that equation (\ref{eq:symmetry of killing form}) implies that an antisymmetric $2$-cocycle vanishes on pairs of opposite roots. In terms of the Lie algebra model of Corollary \ref{cor:model} equation (\ref{eq:symmetry of killing form}) is the symmetry of the Killing form.
\begin{Lemma}
\label{lem:cocycle zero on positive roots}
If \(\omega^2\in Z^2(\roots,M)\) is such that 
$\omega^2(\alpha \ibr \beta)=0$ if $\alpha, \beta\in\roots^+$ or $\beta=\rootm\alpha$, then $\omega^2=0$. 
\end{Lemma}
\begin{proof} 
We first show that $\omega^2(\alpha \ibr \beta)=0$ if $\alpha\rootp\beta\in\roots^+_0$. This holds by assumption when $\alpha,\beta\in\roots^+$ or $\alpha\rootp\beta=0$. If on the other hand $\alpha\in\roots^-$ that is, if \(\alpha\) is a negative root, then $\beta\in\roots^+$ and equation (\ref{eq:identity 1}) shows 
$\omega^2(\alpha \ibr \beta)=\omega^2(\alpha \ibr \rootm\alpha)-\omega^2(\rootm\alpha \ibr \alpha\rootp\beta)=0$ since the components of the argument 
$\lbr\alpha \ibr \rootm\alpha\rbr$ are opposite and the components of the argument $\lbr\rootm\alpha \ibr \alpha\rootp\beta\rbr$ are both positive.
Now suppose $\alpha\in\roots^+$ and $\beta\in\roots^-$. Then equation (\ref{eq:identity 2}) shows $\omega^2(\alpha \ibr \beta)=0$, and we conclude 
$\omega^2(\alpha \ibr \beta)=0$ if $\alpha\rootp\beta\in\roots^+_0$. \\
Using this information we see from equation (\ref{eq:identity 1}) that $\omega^2(\alpha \ibr \beta)=0$ if $\beta\in\roots^+$ and from (\ref{eq:identity 2}) that $\omega^2(\alpha \ibr \beta)=0$ if $\alpha\in\roots^+$. That leaves only one case to check: $\alpha, \beta\in\roots^-$, which again follows readily from either of these equations.
\end{proof}

\begin{Theorem}
\label{thm:H^2_+=0}
Let $\omega^2_+\in Z^2_+(\roots,M)$. Define $\omega^1\in C^1(\roots,M)$ inductively on the height of the roots as follows.
\begin{enumerate}
\item The values on simple roots, $\omega^1(\alpha),\;\alpha\in\sroots$, are chosen as free variables, in $M$.
\item If $\beta\in\roots^+$ there exists a simple root $\alpha$ such that $\beta\rootm\alpha\in\roots_0$ and we define \[\omega^1(\beta)=\omega^1(\beta\rootm\alpha)-\omega^2_+(\alpha \ibr \beta\rootm\alpha)+\omega^1(\alpha).\]
\item If $\beta\in\roots^-$ set $\omega^1(\beta)=\omega^2_+(\beta \ibr \rootm\beta)-\omega^1(\rootm\beta)$.
\end{enumerate}
Then $\omega^1$ is well defined and $\mathsf{d}^1\omega^1=\omega^2_+$. In particular, $H^2_+(\roots,M)$ is trivial.
\end{Theorem}
\begin{proof}
It is possible that there are multiple ways to write a positive root $\beta$ of height $h+1$ as a sum of roots of height $h$ and $1$. Therefore one needs to show that $\omega^1$ given in the theorem is well defined. Let $\beta=\gamma_i\rootp\alpha_i=\gamma_j\rootp\alpha_j$ be two such decompositions. Then $\delta=\gamma_i\rootm\alpha_j=\gamma_j\rootm\alpha_i=\beta\rootm\alpha_i\rootm\alpha_j$ is a root as well, by Lemma \ref{lem:properties of roots}, statement \ref{item:minus two simple roots}. 

Suppose $\omega^1$ is well defined on roots up to height $h$. Let $\omega^1_i(\beta)=\omega^1(\beta\rootm\alpha_i)-\omega^2_+(\alpha_i \ibr \beta\rootm\alpha_i)+\omega^1(\alpha_i)$. Then
\begin{align*}
\omega^1_i(\beta)-\omega^1_j(\beta)
&=\omega^1(\gamma_i)-\omega^2_+(\alpha_i \ibr \gamma_i)+\omega^1(\alpha_i)-\left[\omega^1(\gamma_j)-\omega^2_+(\alpha_j \ibr \gamma_j)+\omega^1(\alpha_j)\right]\\
&=\omega^1(\delta\rootp\alpha_j)-\omega^2_+(\alpha_i \ibr \gamma_i)+\omega^1(\alpha_i)-\left[\omega^1(\delta\rootp\alpha_i)-\omega^2_+(\alpha_j \ibr \gamma_j)+\omega^1(\alpha_j)\right]\\
&=\omega^1(\delta)-\omega^2_+(\alpha_j \ibr \delta)+\omega^1(\alpha_j)-\omega^2_+(\alpha_i \ibr \gamma_i)+\omega^1(\alpha_i)\\
&\quad-\left[\omega^1(\delta)-\omega^2_+(\alpha_i \ibr \delta)+\omega^1(\alpha_i)-\omega^2_+(\alpha_j \ibr \gamma_j)+\omega^1(\alpha_j)\right]\\
&=-\omega^2_+(\alpha_j \ibr \delta)-\omega^2_+(\alpha_i \ibr \gamma_i)+\omega^2_+(\alpha_i \ibr \delta)+\omega^2_+(\alpha_j \ibr \gamma_j)\\
&\overset{s}{=}\omega^2_+(\delta \ibr \alpha_i)-\omega^2_+(\gamma_i \ibr \alpha_i)+\omega^2_+(\alpha_j \ibr \gamma_j)-\omega^2_+(\alpha_j \ibr \delta)\\
&=\mathsf{d}^2\omega^2_+(\alpha_j \ibr \delta \ibr \alpha_i)=0.
\end{align*}
The symbol $\overset{s}{=}$ marks the place where the required symmetry of $\omega$ is used.

Now we turn to the proof of $\mathsf{d}^1\omega^1(\alpha \ibr \beta)=\omega^2_+(\alpha \ibr \beta)$, where we will use induction on $\text{height}\, \alpha +\text{height}\,\beta$. If $H=\max\{\text{height}\, \alpha +\text{height}\,\beta\;|\;\alpha,\,\beta,\,\alpha\rootp\beta\in \roots^+\}$ then for each $2\le h\le H$ there exists $\lbr\alpha\br\beta\rbr\in T_2(\roots)$ with $\alpha,\,\beta \in \roots^+$ and $\text{height}\, \alpha +\text{height}\,\beta=h$. This follows from Lemma \ref{lem:properties of roots}, statement \ref{item:subtract simple root from sum and summand}. 

First consider two simple roots, $\alpha_i$ and $\alpha_j$. 
By definition $\omega^1(\alpha_j\rootp\alpha_i)=\omega^1(\alpha_j)-\omega^2_+(\alpha_i \ibr \alpha_j)+\omega^1(\alpha_i)$, hence 
\begin{align*}
\mathsf{d}^1\omega^1(\alpha_i \ibr \alpha_j)&=\omega^1(\alpha_j)-\omega^1(\alpha_i\rootp\alpha_j)+\omega^1(\alpha_i)\\
&=\omega^1(\alpha_j)-\left[\omega^1(\alpha_j)-\omega^2_+(\alpha_i \ibr \alpha_j)+\omega^1(\alpha_i)\right]+\omega^1(\alpha_i)\\
&=\omega^2_+(\alpha_i \ibr \alpha_j).
\end{align*}
Suppose that $\mathsf{d}^1\omega^1(\tilde{\alpha} \ibr \tilde{\beta})=\omega^2_+(\tilde{\alpha} \ibr \tilde{\beta})$ if $\tilde{\alpha},\,\tilde{\beta},\,\tilde{\alpha}\rootp\tilde{\beta}\in \roots^+$ 
and $\text{height}\, \tilde{\alpha} +\text{height}\,\tilde{\beta}\le h$. 
Now consider another pair $\alpha,\,\beta,\,\alpha\rootp\beta\in \roots^+$ such that $\text{height}\, \alpha +\text{height}\,\beta= h+1$. 
By Lemma \ref{lem:properties of roots} statement \ref{item:subtract simple root from sum and summand} there is a simple root $\alpha_i$ such that $\alpha\rootp\beta\rootm\alpha_i\in\roots^+$ 
and without loss of generality we may assume $\bar{\alpha}=\alpha\rootm\alpha_i\in\roots_0^+$ 
(notice that a positive root minus a simple root is a nonnegative root, since each root is decomposed into simple roots with all positive or all negative coefficients). 
In the following calculation we use that, by definition, 
$\omega^1(\alpha_i\rootp\bar{\alpha}\rootp\beta)=\omega^1(\bar{\alpha}\rootp\beta)-\omega^2_+(\alpha_i \ibr \bar{\alpha}\rootp\beta)+\omega^1(\alpha_i)$ and $\omega^1(\alpha_i\rootp\bar{\alpha})=\omega^1(\bar{\alpha})-\omega^2_+(\alpha_i \ibr \bar{\alpha})+\omega^1(\alpha_i)$, and by induction hypothesis, 
$\omega^2_+(\bar{\alpha} \ibr \beta)=\omega^1(\beta)-\omega^1(\bar{\alpha}\rootp\beta)+\omega^1(\bar{\alpha})$.

We compute
\begin{align*}
\mathsf{d}^1\omega^1(\alpha \ibr \beta)-\omega^2_+(\alpha \ibr \beta)
&=\omega^1(\beta)-\omega^1(\alpha\rootp\beta)+\omega^1(\alpha)-\omega^2_+(\alpha \ibr \beta)\\
&=\omega^1(\beta)-\omega^1(\alpha_i\rootp\bar{\alpha}\rootp\beta)+\omega^1(\alpha_i\rootp\bar{\alpha})-\omega^2_+(\alpha \ibr \beta)\\
&=\omega^1(\beta)-\left[\omega^1(\bar{\alpha}\rootp\beta)-\omega^2_+(\alpha_i \ibr \bar{\alpha}\rootp\beta)+\omega^1(\alpha_i)\right]\\
&\quad+\left[\omega^1(\bar{\alpha})-\omega^2_+(\alpha_i \ibr \bar{\alpha})+\omega^1(\alpha_i)\right]-\omega^2_+(\alpha \ibr \beta)\\
&=\left[\omega^1(\beta)-\omega^1(\bar{\alpha}\rootp\beta)+\omega^1(\bar{\alpha})\right]+\omega^2_+(\alpha_i \ibr \bar{\alpha}\rootp\beta)-\omega^2_+(\alpha_i \ibr \bar{\alpha})-\omega^2_+(\alpha \ibr \beta)\\
&=\omega^2_+(\bar{\alpha} \ibr \beta)-\omega^2_+(\alpha \ibr \beta)+\omega^2_+(\alpha_i \ibr \bar{\alpha}\rootp\beta)-\omega^2_+(\alpha_i \ibr \bar{\alpha})\\
&=\mathsf{d}^2\omega^2(\alpha_i \ibr \bar{\alpha} \ibr \beta)=0,
\end{align*}
thus $\mathsf{d}^1\omega^1-\omega^2_+$ is zero on $(\roots^+\br\roots^+)\cap T_2(\roots)$ and it also maps $\lbr\alpha \ibr \rootm\alpha\rbr$ to zero for all $\alpha\in\roots$. Lemma \ref{lem:cocycle zero on positive roots} then implies $\mathsf{d}^1\omega^1=\omega^2_+$.
\end{proof}

\def\freex{\xi_2}
\def\freey{\xi_1}
\subsection{Some illustrative examples}
\label{sec:examples}
From a computational point of view \cite{MR2718933,KLS15}, one would very much like to work over \(\Mon[\mb{I}]\) instead of \(\Mon(\mb{I})\). 
The analogue of Theorem \ref{thm:H^2_+=0}, however, does not hold,
as the following example shows, which also suggests possible alternative solutions to this problem by extension of the module. We change the notation back to multiplicative,
to stay close to the application we have in mind.
Notice that \(\Mon{[}\mb{I}{]}\) is not a \(\mb{Z}\)-module in the sense of Section \ref{sec:CohomRoot}, since there are no multiplicative inverses (where, just to be very clear on this,
multiplicative here means \(\mb{Z}\)-additive).

\begin{Example}[$H^2_+(B_2,\Mon{[}\mb{I}{]})\ne 1$]\label{ex:B2}
Consider the root system $B_2$. In contrast to $A_2$, in which all roots have the same length, in $B_2$  the roots admit two lengths, and thus are either \emph{short} or \emph{long}
(see the Figure \ref{fig:B2}).
Define $\omega^1\in C^1(B_2,\Mon{[}\mb{I}{]})$ by 
\[
\omega^1(\alpha)=\left\{
\begin{array}{ll}\mb{I}&\text{if $\alpha$ is a short root,}\\
1&\text{if $\alpha$ is a long root.}\end{array}\right.
\]
The coboundary $\mathsf{d}^1\omega^1$ takes values in $\{1,\mb{I}^2\}$. Therefore we can define a $2$-cocycle $\omega^2_+$ by \[\omega^2_+=\sqrt{\mathsf{d}^1
\omega^1}\in Z^2(B_2,\Mon{[}\mb{I}{]}).\] 
More explicitly,
\[\omega^2_+(\alpha \ibr \beta)=\left\{
\begin{array}{ll}
\mb{I}&\text{if $\alpha$ and $\beta$ are short,}\\
1&\text{otherwise.}\end{array}\right.\]
We integrate $\omega^2_+$ using the algorithm of Theorem \ref{thm:H^2_+=0}. Let $\alpha_1$ be the long simple root and $\alpha_2$ the short simple root of $B_2$. 
We take the values of $\varpi^1$ on these simple roots as free variables (using the kernel of $\mathsf{d}^1$ in full), $\varpi^1(\alpha_1)=\freey$ and $\varpi^1(\alpha_2)=\freex$. 
Then $\varpi^1(\alpha_1\rootp\alpha_2)=\varpi^1(\alpha_1)\omega^2_+(\alpha_2 \ibr \alpha_1)^{-1}\varpi^1(\alpha_2)=\freey\freex$, and $\varpi^1((\alpha_1\rootp\alpha_2)\rootp\alpha_2)=\varpi^1(\alpha_1\rootp\alpha_2)\omega^2_ +(\alpha_2 \ibr \alpha_1\rootp\alpha_2)^{-1}\varpi^1(\alpha_2)=\freey\freex^2\mb{I}^{-1}$.
\begin{figure}[H]
\caption{A symmetric $2$-cocycle which is not exact over $\Mon[\mb{I}]$. Links sent to $\mb{I}$ and $1$ are depicted in \colorp\;and \colorchain\;respectively. }
\label{fig:B2}
\begin{center}
\begin{tikzpicture}[scale=2]
\path 	
node at (1,-1) [root,draw,label=0: {$ \varpi^1(\alpha_1)=\freey $}] (m10) {$  $}
node at (0,1) [root,draw,label=90: {$ \varpi^1(\alpha_2)=\freex $}] (m01) {$  $}
node at (1,0) [root,draw,label=0: {$ \freey\freex $}] (m11) {$  $}
node at (1,1) [root,draw,label=0: {$ \mb{I}^{-1}\freey\freex^2 $}] (m12) {$  $}
node at (-1,1) [root,draw,label=180: {$ \freey^{-1} $}] (n10) {$  $}
node at (0,-1) [root,draw,label=270: {$ \mb{I}\freex^{-1} $}] (n01) {$  $}
node at (-1,0) [root,draw,label=180: {$ \mb{I}(\freey\freex)^{-1} $}] (n11) {$  $}
node at (-1,-1) [root,draw,label=180: {$ \mb{I}(\freey\freex^2)^{-1} $}] (n12) {$  $}
;
\draw[schain](m10)to node{}(n10) ;
\draw[schain](m12)to node{}(n12) ;
\draw[schain](m01)to node{}(n12) ;
\draw[schain](m11)to node{}(n12) ;
\draw[schain](n01)to node{}(m12) ;
\draw[schain](n11)to node{}(m12) ;
\draw[schain](m01)to node{}(m10) ;
\draw[schain](n11)to node{}(m10) ;
\draw[schain](n01)to node{}(n10) ;
\draw[schain](m11)to node{}(n10) ;

\draw[scochainp](m01) to node{}(m11);
\draw[scochainp](m11) to node{}(n01);
\draw[scochainp](n01) to node{}(n11);
\draw[scochainp](n11) to node{}(m01);
\draw[scochainp](m01) to node{}(n01) ;
\draw[scochainp](m11)to node{}(n11) ;
\end{tikzpicture}
\end{center}
\end{figure}
This example is of particular interest because the $2$-cocycle $\omega^2_+$ takes values in $\Mon[\mb{I}]$ but it cannot be integrated in this \(\mb{N}_0\)-module. 
Indeed, $\varpi^1$ takes values $\freey$ and $\freey^{-1}$, hence requiring nonnegative powers forces $\freey=1$. 
The other opposite long roots then have values $(\mb{I}^{-1}\freex^2)^{\pm 1}$, hence avoiding negative powers requires $\freex=\sqrt{\mb{I}}$. 
We find that there is a unique solution to $\omega^2_+=\mathsf{d}^1\varpi^1$ with values in $\Mon[\mb{\sqrt{I}}]$, namely $\varpi^1=
\sqrt{\omega^1}$.
Notice that the occurrence of \(\sqrt{\mb{I}}\) is not in conflict with Theorem \ref{thm:H^2_+=0}, since it would be no problem to integrate allowing negative powers,
but a consequence of our abuse of the freedom given by \(\ker\mathsf{d}^1\).
\end{Example}

\begin{Example}[Explicit integration of an element of $Z_+^2(B_2,\Mon{[}\mb{I},\mb{J}{]})$]\label{ex:ALiAB2}
The example is the ALiA obtained by starting with the \(5\)-dimensional irreducible representation \(\mb{Y}_8\) of the icosahedral group $\mb{Y}$, and considering equivariant rational maps taking values in $\mf{so}(\mb{Y}_8)$, with poles in the smallest orbit (cf. \cite{KLS15}). 
The group generator of order five, \(r\) with \(r^5=1\), has been diagonalised. We have computed the ALiA, computed a Cartan-Weyl basis and found
that the structure constants are integer multiples of elements in \(\Mon[\mb{I},\mb{J}]\),  
where \(\mb{I}\) and \(\mb{J}\) are the automorphic functions with poles in the smallest orbit and zeros in the other two exceptional orbits respectively.
We then find the following element of $Z_+^2(B_2,\Mon[\mb{I},\mb{J}])$:
\[
\begin{array}{ll}
\omega_+^2(\alpha_{1} \ibr  \alpha_{2})= 1,
&\omega_+^2(\rootm\alpha_{2} \ibr \rootm(\alpha_{1}\rootp \alpha_{2}))= \mb{I},
\\\omega_+^2(\alpha_{1}\rootp 2\alpha_{2} \ibr \rootm\alpha_{2})= \mb{I},
&\omega_+^2(\alpha_{2} \ibr  \rootm\alpha_{2})= \mb{I},
\\\omega_+^2(\alpha_{2} \ibr  \alpha_{1}\rootp\alpha_{2})= 1,
&\omega_+^2(\alpha_{2} \ibr  \rootm(\alpha_{1}\rootp\alpha_{2}))= \mb{I},
\\\omega_+^2(\alpha_{2} \ibr  \rootm(\alpha_{1}\rootp 2\alpha_{2}))= 1,
&\omega_+^2(\alpha_{1}\rootp\alpha_{2} \ibr \rootm\alpha_{2})= \mb{I},
\\\omega_+^2(\alpha_{1} \ibr  \rootm\alpha_{1})= \mb{J},
&\omega_+^2(\alpha_{1} \ibr  \rootm(\alpha_{1}\rootp\alpha_{2}))=\mb{J},
\\\omega_+^2(\alpha_{1}\rootp\alpha_{2} \ibr  \rootm\alpha_{1})= \mb{J},
&\omega_+^2(\alpha_{1}\rootp 2\alpha_{2} \ibr \rootm(\alpha_{1}\rootp\alpha_{2}))= \mb{IJ},
\\\omega_+^2(\alpha_{1}\rootp\alpha_{2} \ibr \rootm(\alpha_{1}\rootp\alpha_{2}))= \mb{I}\mb{J},
&\omega_+^2(\alpha_{1}\rootp\alpha_{2} \ibr \rootm(\alpha_{1}\rootp 2\alpha_{2}))= \mb{J},
\\\omega_+^2(\alpha_{1}\rootp 2\alpha_{2} \ibr \rootm(\alpha_{1}\rootp 2\alpha_{2}))= \mb{I}\mb{J},
&\omega_+^2(\rootm\alpha_{1} \ibr  \rootm\alpha_{2})= 1.
\end{array}
\]
This is integrated to 
\[
\begin{array}{ll}
\text{long}&\text{short}
\\\hline
\omega^1(\alpha_1)=1,&\omega^1(\alpha_2)=1,\\
\omega^1(\rootm\alpha_1)=\mb{J},&\omega^1(\rootm\alpha_2)=\mb{I},\\
\omega^1(\alpha_1\rootp 2\alpha_2)=1,&\omega^1(\alpha_1\rootp\alpha_2)=1,\\
\omega^1(\rootm(\alpha_1\rootp 2\alpha_2))=\mb{IJ},&\omega^1(\rootm(\alpha_1\rootp\alpha_2))=\mb{IJ}.
\end{array}
\]
The short count (summing over the Killing forms \(K_{\tilde{\mf{g}}}(\tilde{e}_{\rootp\alpha}, \tilde{e}_{\rootm\alpha})\), \(\alpha=\alpha_2, \alpha_1\rootp\alpha_2\)) 
is 
\(\mb{I}+\mb{IJ}\) (cf. \cite{KLS15}), the long count (\(\alpha=\alpha_1,\alpha_1\rootp 2\alpha_2\)) is \(\mb{J}+\mb{IJ}\). 
The total count ($3\mb{I}$s and $3\mb{J}$s) is in accordance with the predictions
given by the codimensions of the invariants subspaces under the conjugating action of the generators of the group on \(\mf{so}_5\).
\end{Example}
In the theory of ALiAs several models of the ALiA play a role: the invariant matrices, the matrices of invariants and the integrated model as the one above.
The matrices of invariants, once in Weyl-Chevalley normal form, are always natural monomial in \(\mb{I}\) and \(\mb{J}\), and they can be used in the subsequent search for integrable systems,
which was the original motivation for ALiAs (see e.g. \cite{Lombardo,lm_jpa04,lm_cmp05,Bury2010}). 
The integrated models cannot in general be used for this purpose, unless they are natural, but they play a role in 
establishing whether two given ALiAs are isomorphic or not, 
and in the choice of a normal form for isomorphic cases.

\begin{Example}[Lack of root symmetry]
All previous examples of $2$-cocycles were left unchanged by some nontrivial automorphism of the root system.
In Figure \ref{fig:asymmetric cocycle} we present a $2$-cocycle $\omega^2$ on $G_2$ with values $1$ and $\mb{I}$ (in additive notation, $\omega^2\in Z^2_+(\roots,\{0,1\})$) which has no root symmetry: the only automorphism of the root system preserving this cocycle is the identity. This proves that the closedness condition does not impose any symmetry.
\begin{figure}[H]
\caption{A symmetric $2$-cocycle without root symmetry. Links sent to $\mb{I}$ are shown in \colorp\;and links sent to $1$ are not shown. }
\label{fig:asymmetric cocycle}
\begin{center}
\begin{tikzpicture}[scale=1.7]
\path 
node at (1,0) [root,draw,label=-60: $\alpha_1 $] (m10) {$ $}
node at (-1.5,.866) [root,draw,label=90: $\alpha_2 $] (m01) {$ $}
node at (-.5,.866) [root,draw,label=90: $ $] (m11) {$ $}
node at (.5,.866) [root,draw,label=90: $ $] (m21) {$ $}
node at (1.5,.866) [root,draw,label=90: $ $] (m31) {$ $}
node at (0,2*.866) [root,draw,label=90: $ $] (m32) {$ $}
node at (-1,0) [root,draw,label=180: $ $] (n10) {$ $}
node at (1.5,-.866) [root,draw,label=270: $ $] (n01) {$ $}
node at (.5,-.866) [root,draw,label=270: $ $] (n11) {$ $}
node at (-.5,-.866) [root,draw,label=270: $ $] (n21) {$ $}
node at (-1.5,-.866) [root,draw,label=270: $ $] (n31) {$ $}
node at (0,-2*.866) [root,draw,label=270: $ $] (n32) {$ $}		
;
\draw[scochainp] (n10) to node []{$ $} (m11);
\draw[scochainp] (n10) to node []{$ $} (m21);
\draw[scochainp] (n10) to node []{$ $} (m31);
\draw[scochainp] (n10) to node []{$ $} (n01);
\draw[scochainp] (n10) to node []{$ $} (n11);
\draw[scochainp] (n10) to node []{$ $} (n21);
\draw[scochainp] (n11) to node []{$ $} (n21);
\draw[scochainp] (n11) to node []{$ $} (m32);
\draw[scochainp] (n11) to node []{$ $} (m21);
\draw[scochainp] (n21) to node []{$ $} (m31);
\draw[scochainp] (n21) to node []{$ $} (m32);
\draw[scochainp] (m11) to node []{$ $} (n01);
\draw[scochainp] (n01) to [bend left=30](n31);
\draw[scochainp] (n01) to [bend right=30](m32);
\draw[scochainp] (n31) to [bend left=30](m32);
\draw[scochainp] (m10) to (n10);
\draw[scochainp] (m01) to (n01);
\draw[scochainp] (m11) to (n11);
\draw[scochainp] (m21) to (n21);
\draw[scochainp] (m31) to (n31);
\draw[scochainp] (m32) to (n32);
\end{tikzpicture}
\end{center}
\end{figure}
\newcommand{\res}{\textrm{res}\,}
To see that Figure \ref{fig:asymmetric cocycle} indeed depicts a cocycle, one can consider the following construction. Define an additive function $\omega^1$ on the root system by $\omega^1(\alpha_1)=\omega^1(\alpha_2)=1\mod 6$ (where $6$ is the Coxeter number of $G_2$). Then $\omega^2(\alpha,\beta)$ is $\mb{I}$ to the power $\frac{1}{6}(\res\omega^1(\beta)-\res\omega^1(\alpha\rootp\beta)+\res\omega^1(\alpha))$, where $\res:\mb{Z}/6\mb{Z}\rightarrow\{0,1,2,3,4,5\}$ is the residue. This construction produces a $2$-cocycle because it is the boundary of the $1$-cochain $\res\omega^1$. It shows moreover that $\omega^2$ is closely related to the grading of the affine Kac-Moody algebra of type $G_2^{(1)}$ obtained by assigning degree $1$ to the positive Chevalley generators $e_0, e_1, e_2$ and degree $0$ to the Cartan subalgebra. This grading is sometimes called principal, or canonical. Of course, any grading of a Kac-Moody algebra of type $X_\ell^{(1)}$ defined by assigning degree $s_i\in\mb{Z}$ to Chevalley generator $e_i$ and degree $0$ to the Cartan subalgebra produces by this construction a $2$-cocycle on the root system of type $X_\ell$ with values $1$ and $\mb{I}$.

The cocycle of Figure \ref{fig:asymmetric cocycle} is expected to describe part of the structure of an ALiA based on $G_2$ with dihedral symmetry. These ALiAs are however not yet computed. 
\end{Example}

\section{Diagonal contractions and Cartan filtrations}
\label{sec:diagonal contractions and cartan filtrations}
\newcommand{\mc}[1]{\mathcal{#1}}
\newcommand{\GL}{\mathrm{GL}}
\newcommand{\Id}{\mathrm{Id}}
\newcommand{\diag}{\mathrm{diag}}
\newcommand{\sq}{0.866}

The root system cohomology is introduced in this paper in order to study Automorphic Lie Algebras, but there are applications in other areas of Lie theory. In this section we describe a  
connection between root cohomology and contractions
and filtrations of Lie algebras using the vanishing of the second cohomology group. 
Interestingly, all three mentioned applications are associated to those $2$-cocycles that have no negative values on $2$-links.
This makes the following notational convention convenient.
\begin{Definition}
Let $C^n_\pm(\roots,\mb{N}_0)$ be the forms in $C^n_\pm(\roots,\mb{Z})$ whose values on $T_n(\roots)$ belong to $\mb{N}_0$, and $Z^n_\pm(\roots,\mb{N}_0)=Z^n_\pm(\roots,\mb{Z})\cap C^n_\pm(\roots,\mb{N}_0)$.
Likewise, let $C^n(\roots,\{0,1\})$ be the forms in $C^n(\roots,\mb{Z})$ whose values on $T_n(\roots)$ belong to $\{0,1\}$ and $Z^n_\pm(\roots,\{0,1\})=Z^n_\pm(\roots,\mb{Z})\cap C^n_\pm(\roots,\{0,1\})$. Let $C^n(\roots,\mb{R}_{\ge 0})$ be the forms in $C^n(\roots,\mb{R})$ whose values on $T_n(\roots)$ belong to $\mb{R}_{\ge 0}$ and $Z^n_\pm(\roots,\mb{R}_{\ge 0})=Z^n_\pm(\roots,\mb{R})\cap C^n_\pm(\roots,\mb{R}_{\ge 0})$. 
\end{Definition}
\subsection{Diagonal contractions}
In this section we write a Lie algebra as a pair $\mf{g}=(V,\mu)$ where $V$ is a complex finite dimensional vector space and $\mu$ is a Lie bracket on $V$. Notice that $\mu$ is an element of the submanifold $\mc{L}$ of $V^\ast\otimes_{\mb{C}} V^\ast\otimes_{\mb{C}} V$ consisting of all Lie brackets on $V$, and $\GL(V)$ acts on this manifold in a canonical way 
\[(U\cdot\mu)(x,y)= U\mu(U^{-1}x,U^{-1}y),\quad U\in\GL(V),\,\mu\in\mc{L},\, x,y\in V.\] The orbits of this action are the manifolds of equivalent Lie brackets.

\begin{Definition}
A Lie algebra $\mf{g}^0=(V,\mu^0)$ is a \emph{contraction} of $\mf{g}=(V,\mu)$ if $\mu^0$ is an element of the closure (Zarisky or Euclidean) of the orbit $\GL(V)\mu$, and we write $\mf{g}\rightarrow\mf{g}^0$, or, when $V$ is understood, $\mu\rightarrow\mu^0$. 
\end{Definition}
For example, if $z\mapsto U_z$ defines a continuous curve $(0,1]\rightarrow \GL(V)$ and the limit $\mu^0=\lim_{z\rightarrow 0^+}U_z\cdot \mu$ exists (all the structure constants have finite limit), then $(V,\mu^0)$ is a contraction of $(V,\mu)$. 
Any Lie algebra has an abelian contraction, as is shown by taking $U_z=z\Id$. 
As a second, more general example, consider the curve $z\mapsto D_z=\diag(z^{-r_1},\ldots,z^{-r_n})$ acting on basis $\{x_1,\ldots,x_k\}$ of $\mf{g}$. Let $C^k_{ij}$ be the structure constants defining $\mu$ with respect to this basis, and $D_z\cdot C^k_{ij}$ be the structure constants defining $D_z\cdot\mu$ . Then 
\begin{equation}
\label{eq:action on structure constants}
D_z\cdot C^k_{ij}= z^{r_i+r_j-r_{k}}C^k_{i,j}.
\end{equation}
We see that nonzero structure constants can become zero in the contraction, but not the other way around. 
\begin{Definition}
Two contractions $\mf{g}\rightarrow\mf{g}^0$ and $\tilde{\mf{g}}\rightarrow\tilde{\mf{g}}^0$ are called \emph{equivalent} if there are Lie algebra isomorphisms $\mf{g}\cong\tilde{\mf{g}}$ and $\mf{g}^0\cong\tilde{\mf{g}}^0$.
\end{Definition}
\begin{Definition}
A contraction is called \emph{diagonal} if it can be realised by a curve $z\mapsto U_z=A D_z B$ where $A, B\in\GL(V)$ and $D_z=\diag(f_1(z),\ldots,f_k(z))$ and $f_i:(0,1]\rightarrow \mb{C}\setminus\{0\}$. A diagonal contraction is called \emph{generalised In\"{o}n\"{u}-Wigner} if one can choose $f_i(z)=z^{r_i}$ with $r_i$ in $\mb{R}$. The tuple $(r_1,\ldots,r_k)$ is called the \emph{signature} of the curve. 
\end{Definition}
When considering diagonal contractions up to equivalence, one can forget about $A$ and $B$. If $U_z=A D_z B$ defines a contraction $\mu\rightarrow\mu^0$, then $D_z$ defines the equivalent contraction $B\cdot\mu\rightarrow A^{-1}\cdot\mu^0$. A more substantial simplification is given by the following.
\begin{Theorem}[\cite{popovych2009equivalence}]
\label{thm:diagonal contractions are iw-contractions}
Any diagonal contraction is equivalent to a generalised In\"{o}n\"{u}-Wigner contraction that can be realised by a curve with integer signature.
\end{Theorem} 
The cocycles $Z^2_+(\roots,\mb{R}_{\ge 0})$ are related to contractions as follows.
Consider a diagonalisable curve $z\mapsto D_z$ with real powers of $z$ as eigenvalues, fixing a Cartan subalgebra $\mf{h}$ of a simple Lie algebra $\mf{g}$ pointwise. Then $D_z$ commutes with the adjoint action of $\mf{h}$ hence there exists a basis for $V$ diagonalising $D_z$ and $\ad(\mf{h})$ simultaneously. Since root spaces of simple Lie algebras are one-dimensional, we obtain a map $\omega^1\in C^1(\roots,\mb{R})$ defined by $\omega(\alpha)=r$ if $D_z$ acts on $\mf{g}_\alpha$ by multiplication with $z^{-r}$.
From (\ref{eq:action on structure constants}) we can see that the Lie bracket $D_z\cdot \mu$ corresponds to (\ref{eq:genCartanWeyl}) defined by $\omega^2_+=\mathsf{d}^1\omega^1$ when the signature is integer. The limit $\mu^0=\lim_{z\mapsto 0}D_z\cdot \mu$ exists if and only if $\omega^2_+(\alpha,\beta)\ge 0$ for all $\lbr\alpha \br \beta\rbr\in T_2(\roots)$, i.e. $\omega^2_+\in Z^2_+(\roots,\mb{R}_{\ge 0})$. Moreover, it then only depends on the kernel of $\omega^2_+$.

We can also go in the other direction thanks to the acyclicity $H^2_+(\roots,\mb{R})=0$: any cocycle $\omega^2_+\in Z^2_+(\roots,\mb{R}_{\ge 0})$ defines a contraction of the simple Lie algebra $\mf{g}$ as follows. Let $\omega^1\in C^1(\roots,\mb{R})$ be such that $\mathsf{d}^1\omega^1=\omega^2_+$.  Define $D_{z}\in\GL(\mf{g})$ to be the identity on $\mf{h}$ and multiplication by $z^{-\omega^1(\alpha)}$ on $\mf{g}_{\alpha}$ for $\alpha\in\roots$. Then $D_z$ defines a contraction of $\mf{g}$ because $\omega^2_+(\alpha,\beta)\ge 0$ for all $\lbr\alpha \br \beta\rbr\in T_2(\roots)$.

We have found a relation between the kernels of elements of $Z^2_+(\roots,\mb{R}_{\ge 0})$ and a specific collection of contractions of $\mf{g}$. In the next subsection we present another description of this collection of contractions.

\subsection{Cartan filtrations}
Generalised In\"{o}n\"{u}-Wigner  
contractions with integer signature correspond to integer filtrations.
\begin{Definition}
A $\mb{Z}$-filtration $\{V_i\,|\,i\in\mb{Z}\}$ of a Lie algebra $\mf{g}=(V,\mu)$ is a collection of nested linear subspaces 
$\ldots V_{-2}\subset V_{-1}\subset V_0\subset V_1\subset V_2\ldots$
of $V$ such that 
\begin{equation}
\label{eq:filtration}
\mu(V_i,V_j)\subset V_{i+j}.
\end{equation}
The filtration is bounded if there exists $p$ and $q$ such that $V_p=\{0\}$ and $V_q=V$. 
\end{Definition}
Throughout this section, a filtration will be a bounded $\mb{Z}$-filtration.
Notice that $(V_0,\mu)$ is a Lie algebra and each $V_i$ is a $(V_0,\mu)$-module.

To each filtration is associated a $\mb{Z}$-graded Lie algebra $\mf{g}^0=(V,\mu^0)$ as follows. We have $V\cong\bigoplus_i W_i$ (vector space direct sum) where $W_i=V_i/V_{i-1}$. This isomorphism will be suppressed in the notation. Let $\pi_i:V_i\rightarrow W_i$ be the canonical projection. Notice that $\pi_i(V_j)$ is zero if $j<i$ and not defined if $j>i$. Define
\begin{equation}
\label{eq:bracket graded Lie algebra}
\mu^0(\pi_i{x},\pi_j{y})=\pi_{i+j}{\mu(x,y)}.
\end{equation}
The filtration condition (\ref{eq:filtration}) ensures that this bracket is well defined and graded, $\mu^0(W_i,W_j)\subset W_{i+j}$, and $\mf{g}^0$ is a $\mb{Z}$-graded Lie algebra.
\begin{Lemma}[{\cite[Thm 1.2]{grunewald1988varieties}}]
\label{lem:contractions and filtrations}
If $\mf{g}\rightarrow\mf{g}^0$ is a generalised In\"{o}n\"{u}-Wigner contraction with integer signature, then $\mf{g}^0$ is the graded Lie algebra associated to a filtration of $\mf{g}$.

A filtration $\{V_i\,|\,i\in\mb{Z}\}$ realising $\mf{g}^0$ is defined from a curve $z\mapsto D_z$ realising the contraction $\mf{g}\rightarrow\mf{g}^0$ by setting $U_i$ to be the eigenspace of $D_z$ with eigenvector $z^{-i}$ and defining $V_i=\sum_{j\le i} U_j$.

If $\mf{g}^0$ is the graded Lie algebra associated to a filtration of $\mf{g}$, then $\mf{g}\rightarrow\mf{g}^0$ is a generalised In\"{o}n\"{u}-Wigner contraction with integer signature.

A curve  $z\mapsto D_z$ realising a contraction $\mf{g}\rightarrow\mf{g}^0$  is defined from a filtration $\{V_i\,|\,i\in\mb{Z}\}$ realising $\mf{g}^0$
by letting $D_z$ to act on $W_i=V_i/V_{i-1}$ by multiplication with $z^{-i}$ and using the linear isomorphism $V\cong \bigoplus_i W_i$ to define the action of $D_z$ on $V$.
\end{Lemma}
\begin{Definition}
A filtration $\{V_i\,|\,i\in\mb{Z}\}$ of a Lie algebra $\mf{g}=(V,\mu)$ is called a Cartan filtration if $V_0$ contains a Cartan subalgebra of $\mf{g}$.
\end{Definition}
In \cite{barnea2006filtrations} it is shown that the Cartan filtrations of a simple Lie algebras of classical type include their maximal filtrations, which in turn are used to classify maximal graded subalgebras of affine Kac-Moody algebras. Our current interest is with Cartan filtrations themselves.

To each Cartan filtration $\{V_i\,|\,i\in\mb{Z}\}$ of a simple Lie algebra
$\mf{g}$, with Cartan subalgebra $\mf{h}\subset V_0$ defining root spaces $V_\alpha$, $\alpha\in\roots$, we associate an element of $\omega^1\in C^1(\roots,\mb{Z})$ defined by  
\begin{equation}
\label{eq:degree function}
\omega^1(\alpha)=\min\{i\in\mb{Z}\,|\,V_\alpha\subset V_i\}
\end{equation}
called the degree function in \cite{barnea2006filtrations}.
\begin{Lemma}
\label{lem:filtration defines degree function}
The degree function associated to a Cartan filtration $\{V_i\,|\,i\in\mb{Z}\}$ is independent of the choice of Cartan subalgebra of $\mf{g}$ in $V_0$. Thus the filtration uniquely determines the degree function.
\end{Lemma}
\begin{proof}
Let $\mf{h}$ and $\tilde{\mf{h}}$ be Cartan subalgebras of $\mf{g}$ in $V_0$. Then they are also Cartan subalgebras of the Lie algebra $V_0$. Indeed, they are nilpotent and self-normalising since $\mf{h}\subset N_{V_0}(\mf{h})\subset N_{\mf{g}}(\mf{h})=\mf{h}$, idem for $\tilde{\mf{h}}$. Consequently, $\tilde{\mf{h}}=g\mf{h}$ where g is an inner automorphism $g$ of $V_0$. That is, $g$ is a product with factors $e^{\ad(x)}$ where $x\in V_0$. Since $\ad(V_0)V_i\subset V_i$ we get $g V_i=V_i$.

If $\alpha$ is a root of $\mf{h}$, then $\tilde{\alpha}=\alpha\circ g^{-1}$ is a root of $\tilde{\mf{h}}$ and this defines an isomorphism between the root systems for $\mf{h}$ and $\tilde{\mf{h}}$. The root spaces satisfy $V_{\tilde{\alpha}}=g V_\alpha$. Therefore  $V_{\tilde{\alpha}}=g V_\alpha\subset V_i$ if and only if $V_\alpha\subset g^{-1}V_i=V_i$.

Let $\tilde{\omega}^1$ be the degree function defined by the Cartan subalgebra $\tilde{\mf{h}}$. Then $\tilde{\omega}^1(\tilde{\alpha})=\min\{i\in\mb{Z}\,|\,V_{\tilde{\alpha}}\subset V_i\}=\min\{i\in\mb{Z}\,|\,V_{{\alpha}}\subset V_i\}=\omega^1(\alpha)$ and this finishes the proof.
\end{proof}
The next lemma from \cite{barnea2006filtrations} is paraphrased using some of the terminology introduced in this paper.
\begin{Lemma}[{\cite[Lemma 3.1]{barnea2006filtrations}}]
\label{lem:barnea}
Let $\{V_i\,|\,i\in\mb{Z}\}$ be a Cartan filtration of the simple Lie algebra
$\mf{g}$, $\mf{h}\subset V_0$ a Cartan subalgebra defining root spaces $V_\alpha$, $\alpha\in\roots$, and $\omega^1$ the degree function of the filtration.
\begin{enumerate}
\item For each $i\in\mb{Z}$, we have $V_i=\sum V_\alpha$, where the sum is over all $\alpha\in\roots$ with $\omega^1(\alpha)\le i$. In particular, $\omega^1$ uniquely determines the filtration.
\item The boundary $\mathsf{d}^1\omega^1$ takes no negative values on $ T_2(\roots)$, i.e. $\mathsf{d}^1\omega^1\in Z^2_+(\roots,\mb{N}_{0})$.
\end{enumerate}
\end{Lemma}
\begin{Theorem}
\label{thm:cartan filtrations and 1-cochains}
Let $\mf{g}$ be a simple Lie algebra with root system $\roots$. The degree function (\ref{eq:degree function}) defines a $1-1$ relation between Cartan filtration of $\mf{g}$ and the subset of $C^1(\roots,\mb{Z})$ that differentiates into $C^2(\roots,\mb{N}_0)$.
\end{Theorem}
\begin{proof}
Let $\int C^2(\roots,\mb{N}_0)$ denote the subset of $C^1(\roots,\mb{Z})$ that differentiates into $C^2(\roots,\mb{N}_0)$.
Consider the map from the set of Cartan filtrations to $C^1(\roots,\mb{Z})$ sending a filtration to its degree function. This map is well defined by Lemma \ref{lem:filtration defines degree function}. It is injective by the first item of Lemma \ref{lem:barnea}.  
Its image is contained in $\int C^2(\roots,\mb{N}_0)$ by the second item of Lemma \ref{lem:barnea}. Finally, notice that $\int C^2(\roots,\mb{N}_0)$ is contained in the image, since for any $\omega^1\in\int C^2(\roots,\mb{N}_0)$, the spaces $V_i=\sum_{\omega^1(\alpha)\le i} V_\alpha$ satisfy (\ref{eq:filtration}) and $\mf{h}\subset V_0$ hence define a Cartan filtration of $\mf{g}$.
\end{proof}

The graded algebra $\mf{g}^0$  associated to a Cartan filtration $\{V_i\,|\,i\in\mb{Z}\}$ is defined by $\mathsf{d}\omega^1$. In fact, it is enough to know the subset of $T_2(\roots)$ on which $\mathsf{d}\omega^1$ vanishes. We will call this the kernel of $\mathsf{d}\omega^1$.
Indeed, (\ref{eq:bracket graded Lie algebra}) reads $\mu^0(\pi_{\omega^1(\alpha)}e_\alpha,\pi_{\omega^1(\beta)}e_\beta)=\pi_{\omega^1(\alpha)+\omega^1(\beta)}\mu(e_\alpha,e_\beta)$. 
Since $\pi_i V_j=\{0\}$ if $j<i$ we get
\begin{equation}
\label{eq:contraction from 2-cocycle}
\mu^0(e_\alpha,e_\beta)=\left\{\begin{tabular}{ll} $\mu(e_\alpha,e_\beta)$& if $\mathsf{d}\omega^1=0$,\\$0$& if $\mathsf{d}\omega^1>0$\end{tabular}\right .
\end{equation}
(if one does not mind a detour into contractions, one can also see from Lemma \ref{lem:contractions and filtrations} and equation (\ref{eq:action on structure constants}) that 
$
\mu^0({e}_\alpha,{e}_\beta)=\lim_{z\rightarrow 0}z^{\mathsf{d}\omega^1(\alpha,\beta)}{\mu({e}_\alpha,{e}_\beta)}
$).
Let 
$\ker Z^2_+(\roots,M)=\{\ker\omega^2_+\,|\,\omega^2_+\in Z^2_+(\roots,M)\}$.
\begin{Theorem}
\label{thm:cartan filtrations and 2-cocycles}
Let $\mf{g}$ be a simple Lie algebra with root system $\roots$.
There is a $1-1$ relation between contractions associated to Cartan filtration of $\mf{g}$, up to equivalence, and 
$\Aut(\roots)$-orbits in $\ker Z^2_+(\roots,\mb{N}_0)$.
\end{Theorem}
\begin{proof}
Let $\mf{g}\rightarrow\mf{g}^0$ be a contraction and suppose there is a Cartan filtration of $\mf{g}$ with associated graded Lie algebra $\mf{g}^0$. By Theorem \ref{thm:cartan filtrations and 1-cochains} this filtration can be identified with its degree function in $C^1(\roots,\mb{Z})$ that differentiates into $ Z^2_+(\roots,\mb{N}_0)$. From equation (\ref{eq:contraction from 2-cocycle}) we see that $\ker \mathsf{d}^1\omega^1$ only depends on $\mf{g}^0$, and therefore is independent of the choice of Cartan filtration made above. Moreover, we see from (\ref{eq:contraction from 2-cocycle}) that nonisomorphic contractions map to elements of $\ker Z^2_+(\roots,\mb{N}_0)$ which are not in the same $\Aut(\roots)$-orbit. Thus we have constructed an injective map from the set of contractions associated to Cartan filtration into $\ker Z^2_+(\roots,\mb{N}_0)/\Aut(\roots)$.

It remains to be shown that each element of $\ker Z^2_+(\roots,\mb{N}_0)/\Aut(\roots)$ occurs this way. This follows from Theorem \ref{thm:H^2_+=0}.
Let $\omega^2_+\in Z^2_+(\roots,\mb{N}_0)$. Since $H^2_+(\roots,\mb{Z})=\{0\}$ there exists $\omega^1\in C^1(\roots,\mb{Z})$ such that $\mathsf{d}^1\omega^1=\omega^2_+$. By Theorem \ref{thm:cartan filtrations and 1-cochains} there is a Cartan filtration with degree function $\omega^1$, and associated a contraction. The map described in the previous paragraph sends this contraction to the $\Aut(\roots)$-orbit of $\ker \omega^2_+$.
\end{proof}

In the theory of Automorphic Lie Algebras the $2$-cocycles that play a central role are $Z^2_+(\roots,\{0,1\})$ \cite{knibbeler2019hereditary}. The next example shows that such cocycles do not exhibit all kernels $\ker Z^2_+(\roots,\mb{N}_0)$, and therefore not all contractions associated to Cartan filtrations.
\begin{Example}[$\ker Z^2_+(\roots,\mb{N}_0)\ne \ker Z^2_+(\roots,\{0,1\})$.]
Let $\omega^2_+\in C^2_+(\roots(A_2),\mb{N}_0)$ be defined by the following picture, where the number of edges denotes the value of $\omega^2$.
\[\begin{tikzpicture}[scale=\scaleAA]
  \path node at ( 0,0) [label=270: $ $] (zero) {$ $}	
	node at ( 4,0) [draw,root,label=0: $ $] (one) {$ $}
  	node at ( 2,4*\sq) [draw,root,label=60: $ $] (two) {$ $}
  	node at ( -2,4*\sq) [draw,root,label=120: $ $] (three) {$ $}
	node at ( -4,0) [draw,root,label=180: $ $] (four) {$ $}
	node at ( -2,-4*\sq) [draw,root,label=240: $ $] (five) {$ $}
	node at ( 2,-4*\sq) [draw,root,label=300: $ $] (six) {$ $};
	\draw [scochainp](one) to node [sloped,above,near end]{} (three);
	\draw [scochainp](one.150) to node [sloped,above,near end]{} (four.30);
	\draw [scochainp](one.-150) to node [sloped,above,near end]{} (four.-30);
	\draw [scochainp](one) to node [sloped,above,near end]{} (five);
	\draw [scochainp](two) to node [sloped,below,near end]{} (five);
	\draw [scochainp](two) to node [sloped,above]{} (four);
	\draw [scochainp](three) to node [sloped,below]{} (six);
	\draw [scochainp](four) to node [sloped,below]{} (six);
\end{tikzpicture}\]
One can check that $\omega^2_+$ is a cocycle by evaluating $\mathsf{d}^2\omega^2_+$ in all the $3$-links which are listed in Example \ref{ex:H^2_-ne0}, or by integrating it.
If $\tilde{\omega}^2_+$ is the element of  $C^2_+(\roots,\{0,1\})$ with $\ker \tilde{\omega}^2_+=\ker {\omega}^2_+$ then $\tilde{\omega}^2_+$ is not a cocycle.
\end{Example}

\section{Conclusions}
We have shown how the cohomology of root systems appears naturally in the theory of ALiAs, once a Cartan-Weyl basis of the ALiA is computed.
The theory of ALiAs, and more specifically, their normal form theory, has been developed in the last decade and this explains why there is no
mention of cohomology of root systems in the literature, since it is the symmetric case that appears in a natural way and there does not exist
an analogous theory in the skew symmetric case, although there are some developments in the theory of Kac-Moody algebras that remind one of such a theory (cf. \cite[p. 105]{MR1104219}).

Even though the application of groupoid cohomology to root systems is new,
it is naturally connected to classical theory of simple Lie algebras, their root systems and representations. For instance, $1$-cocycles $Z^1(\roots,A)$ with values in an abelian group $A$ define gradings of the Lie algebra $\mf{g}=\bigoplus_{a\in A}\mf{g}_a$. In particular,  $Z^1(\roots,\Zn{n})$ corresponds to inner automorphisms $\phi$ with $\phi^n=1$, fixing the Cartan subalgebra. Also, $Z^1(\roots,\mb{Z})$ can be identified with the coweight lattice of the  Lie algebra associated to $\roots$ (using the inner product on the dual of the Cartan subalgebra). Thus, the Weyl group orbits in $Z^1(\roots,\mb{Z})$ correspond to the irreducible representations of the Lie algebra associated to the dual $\roots^\vee$. Moreover, as shown in Section \ref{sec:diagonal contractions and cartan filtrations}, the subset of $C^1(\roots,\mb{Z})$ that differentiates into $C^2(\roots,\mb{N}_0)$ is in $1-1$ correspondence with Cartan $\mb{Z}$-filtrations.
Taking the $\Aut(\roots)$-orbit of the kernel of the derivative of the former, and taking the graded algebra related to the latter results in another $1-1$ correspondence. Thus associating $Z^2_+(\roots,\mb{N}_0)$ to a specific type of contraction.

The results of this work open up a number of interesting problems. Within the theory of ALiAs we solved a first problem in the form of Theorem \ref{thm:H^2_+=0}. That is, when an ALiA is computed and the $2$-cocycle describing its structure obtained, then one can use the integration procedure of Theorem \ref{thm:H^2_+=0} to obtain a concrete model for this Lie algebra.
A second important open problem in the application of root cohomology to the theory of ALiAs, is to classify `binary' $2$-cocycles to advance the classification of ALiAs. 
Indeed, it can be shown that if an ALiA contains a certain Cartan subalgebra, then it is described by a triple of $2$-cocycles $\omega^2\in Z^2_+(\roots,\Mon[\mb{I}])$ with image in $\{1,\mb{I}\}$. Thus one strategy towards a classification of ALiAs is to investigate existence of the Cartan subalgebra and to classify these $2$-cocycles that can only take two distinct values.

There are various open questions regarding integration and differentiation which all have a relevant meaning in the study of ALiAs and their concrete models.
An obvious problem is to find $H^2_+(\roots,\Mon[\mb{I}])$. In other words, what are the conditions under which one can integrate $\omega^2\in Z^2_+(\roots,\Mon[\mb{I}])$ to $\omega^1\in C^1(\roots,\Mon[\mb{I}])$.\\
A further interesting problem is to find out under what conditions $\omega^1\in C^1(\roots,\Mon[\mb{I}])$ is differentiated into $C^2_+(\roots,\Mon[\mb{I}])$, and similarly, under what conditions $\omega^1\in C^1(\roots,\Mon(\mb{I}))$ is differentiated into $C^2_+(\roots,\Mon[\mb{I}])$. Also, it would be interesting to state under what conditions  $\omega^2\in Z^2_+(\roots,\Mon[\mb{I}])$ can be integrated to $\omega^1\in C^1(\roots,\Mon[\mb{I}^{\nicefrac{1}{n}}])$.

Considering the importance of infinite dimensional Lie algebras in both physics and mathematics, we hope that this approach will stimulate future
research and lead to new and interesting
developments.


\def\cprime{$'$}

\end{document}